\documentclass[reqno, 12pt]{amsart}
\pdfoutput=1
\makeatletter
\let\origsection=\section \def\section{\@ifstar{\origsection*}{\mysection}} 
\def\mysection{\@startsection{section}{1}\z@{.7\linespacing\@plus\linespacing}{.5\linespacing}{\normalfont\scshape\centering\S}}
\makeatother        
\usepackage{amsmath,amssymb,amsthm,bm}    
\usepackage{mathabx}\changenotsign  
\usepackage{bbm, accents}  
\usepackage{mathrsfs} 
\usepackage{dsfont} 
\usepackage[babel]{microtype}
\usepackage{xcolor}  	
\usepackage[backref]{hyperref}
\hypersetup{
	colorlinks,
    linkcolor={red!60!black},
    citecolor={green!60!black},
    urlcolor={blue!60!black},
}

\usepackage{bookmark}

\usepackage[abbrev,msc-links,backrefs]{amsrefs} 
\usepackage{doi}

\renewcommand{\PrintDOI}[1]{\doi{#1}}

\usepackage[T1]{fontenc}
\usepackage{lmodern}

\usepackage[english]{babel}
\numberwithin{equation}{section}

\usepackage{geometry}
\usepackage{enumitem}

\def\rmlabel{\upshape({\itshape \roman*\,})}

\def\alabel{\upshape({\itshape \alph*\,})}

\let\polishlcross=\l
\def\l{\ifmmode\ell\else\polishlcross\fi}

\def\qand{\quad\text{and}\quad}
\def\qqand{\qquad\text{and}\qquad}

\let\emptyset=\varnothing
\let\setminus=\smallsetminus

\makeatletter
\def\moverlay{\mathpalette\mov@rlay}
\def\mov@rlay#1#2{\leavevmode\vtop{   \baselineskip\z@skip \lineskiplimit-\maxdimen
   \ialign{\hfil$\m@th#1##$\hfil\cr#2\crcr}}}
\newcommand{\charfusion}[3][\mathord]{
    #1{\ifx#1\mathop\vphantom{#2}\fi
        \mathpalette\mov@rlay{#2\cr#3}
      }
    \ifx#1\mathop\expandafter\displaylimits\fi}
\makeatother

\newcommand{\dcup}{\charfusion[\mathbin]{\cup}{\cdot}}
\newcommand{\bigdcup}{\charfusion[\mathop]{\bigcup}{\cdot}}

\DeclareFontFamily{U}  {MnSymbolC}{}
\DeclareSymbolFont{MnSyC}         {U}  {MnSymbolC}{m}{n}
\DeclareFontShape{U}{MnSymbolC}{m}{n}{
    <-6>  MnSymbolC5
   <6-7>  MnSymbolC6
   <7-8>  MnSymbolC7
   <8-9>  MnSymbolC8
   <9-10> MnSymbolC9
  <10-12> MnSymbolC10
  <12->   MnSymbolC12}{}
\DeclareMathSymbol{\powerset}{\mathord}{MnSyC}{180}

\usepackage{tikz}
\usetikzlibrary{calc}
\pgfdeclarelayer{background}
\pgfdeclarelayer{foreground}
\pgfdeclarelayer{front}
\pgfsetlayers{background,main,foreground,front}

\newlength{\myscaledsize}
 
\newcommand{\vvv}{	\setlength{\myscaledsize}{\the\fontdimen6\font}
	\hspace*{-0.08em}\resizebox{0.5\myscaledsize}{0.4\myscaledsize}{
		\tikz{
			\draw[black,fill=black] (0,0) circle (.2);
			\draw[black,fill=black] (1,0) circle (.2);
			\draw[black,fill=black] (0.5,0.86) circle (.2);}}
}
\newcommand{\VVV}{	\setlength{\myscaledsize}{\the\fontdimen6\font}
	\hspace*{-0.0444em}\raisebox{-0.1\myscaledsize}{\resizebox{0.75\myscaledsize}{0.6\myscaledsize}{
		\tikz{
			\draw[black,fill=black] (0,0) circle (.2);
			\draw[black,fill=black] (1,0) circle (.2);
			\draw[black,fill=black] (0.5,0.86) circle (.2);}}}
			\hspace{0.1em}
}
\newcommand{\pivvv}{\pi_{\vvv}}

\newcommand{\epair}{	\setlength{\myscaledsize}{\the\fontdimen6\font}
	\hspace*{-0.08em}\resizebox{0.5\myscaledsize}{0.4\myscaledsize}{
		\tikz{
			\draw[black,fill=black] (0,0) circle (.2);
			\draw[black,fill=black] (1,0) circle (.2);
						\draw[black] (0.5,0.86) circle (.2);
			\draw[black,line width=4pt ](0,0)--(1,0);}}
}
\newcommand{\EPAIR}{	\setlength{\myscaledsize}{\the\fontdimen6\font}
	\hspace*{-0.0444em}\raisebox{-0.1\myscaledsize}{\resizebox{0.75\myscaledsize}{0.6\myscaledsize}{
		\tikz{
			\draw[black,fill=black] (0,0) circle (.2);
			\draw[black,fill=black] (1,0) circle (.2);
						\draw[black] (0.5,0.86) circle (.2);
			\draw[black,line width=4pt ](0,0)--(1,0);}}}
			\hspace{0.1em}
}

\newcommand{\ev}{	\setlength{\myscaledsize}{\the\fontdimen6\font}
	\hspace*{-0.08em}\resizebox{0.5\myscaledsize}{0.4\myscaledsize}{
		\tikz{
			\draw[black,fill=black] (0,0) circle (.2);
			\draw[black,fill=black] (1,0) circle (.2);
			\draw[black,fill=black] (0.5,0.86) circle (.2);
			\draw[black,line width=4pt ](0,0)--(1,0);}}
}
\newcommand{\EV}{	\setlength{\myscaledsize}{\the\fontdimen6\font}
	\hspace*{-0.0444em}\raisebox{-0.1\myscaledsize}{\resizebox{0.75\myscaledsize}{0.6\myscaledsize}{
		\tikz{
			\draw[black,fill=black] (0,0) circle (.2);
			\draw[black,fill=black] (1,0) circle (.2);
			\draw[black,fill=black] (0.5,0.86) circle (.2);
			\draw[black,line width=4pt ](0,0)--(1,0);}}}
			\hspace{0.1em}
}
\newcommand{\piev}{\pi_{\ev}}

\newcommand{\ee}{	\setlength{\myscaledsize}{\the\fontdimen6\font}
	\hspace*{-0.08em}\resizebox{0.5\myscaledsize}{0.4\myscaledsize}{
		\tikz{
			\draw[black,fill=black] (0,0) circle (.2);
			\draw[black,fill=black] (1,0) circle (.2);
			\draw[black,fill=black] (0.5,0.86) circle (.2);
			\draw[black,line width=4pt ](0,0)--(0.5,0.86);
			\draw[black,line width=4pt ](1,0)--(0.5,0.86);}}
}
\newcommand{\EE}{	\setlength{\myscaledsize}{\the\fontdimen6\font}
	\hspace*{-0.0444em}\raisebox{-0.1\myscaledsize}{\resizebox{0.75\myscaledsize}{0.6\myscaledsize}{
		\tikz{
			\draw[black,fill=black] (0,0) circle (.2);
			\draw[black,fill=black] (1,0) circle (.2);
			\draw[black,fill=black] (0.5,0.86) circle (.2);
			\draw[black,line width=4pt ](0,0)--(0.5,0.86);
			\draw[black,line width=4pt ](1,0)--(0.5,0.86);}}}
			\hspace{0.1em}
}
\newcommand{\piee}{\pi_{\ee}}

\let\pill=\piee

\newcommand{\eee}{	\setlength{\myscaledsize}{\the\fontdimen6\font}
	\hspace*{-0.08em}\resizebox{0.5\myscaledsize}{0.4\myscaledsize}{
		\tikz{
			\draw[black,fill=black] (0,0) circle (.2);
			\draw[black,fill=black] (1,0) circle (.2);
			\draw[black,fill=black] (0.5,0.86) circle (.2);
			\draw[black,line width=4pt ](0,0)--(1,0);
			\draw[black,line width=4pt ](0,0)--(0.5,0.86);
			\draw[black,line width=4pt ](1,0)--(0.5,0.86);}}
}
\newcommand{\EEE}{	\setlength{\myscaledsize}{\the\fontdimen6\font}
	\hspace*{-0.0444em}\raisebox{-0.1\myscaledsize}{\resizebox{0.75\myscaledsize}{0.6\myscaledsize}{
		\tikz{
			\draw[black,fill=black] (0,0) circle (.2);
			\draw[black,fill=black] (1,0) circle (.2);
			\draw[black,fill=black] (0.5,0.86) circle (.2);
			\draw[black,line width=4pt ](0,0)--(1,0);
			\draw[black,line width=4pt ](0,0)--(0.5,0.86);
			\draw[black,line width=4pt ](1,0)--(0.5,0.86);}}}
			\hspace{0.1em}
}

\theoremstyle{plain}
\newtheorem{thm}{Theorem}[section]
\newtheorem{fact}[thm]{Fact}
\newtheorem{prop}[thm]{Proposition}

\newtheorem{lemma}[thm]{Lemma}

\theoremstyle{definition}

\newtheorem{dfn}[thm]{Definition}
\newtheorem{exmp}[thm]{Example}

\newtheorem{prob}[thm]{Problem}
\newtheorem{quest}[thm]{Question}

\newcommand{\seq}[1]{\accentset{\rightharpoonup}{#1}}  
\def\snake{\normalfont\scshape\centering\S}

\let\eps=\varepsilon
\let\theta=\vartheta
\let\rho=\varrho
\let\phi=\varphi

\def\NN{\mathds N}

\def\RR{\mathds R}
\def\PP{\mathds P}
 
\def\cA{{\mathcal A}}

\def\cM{{\mathcal M}}

\def\cB{{\mathcal B}}

\def\cP{{\mathcal P}}

\def\cK{{\mathcal K}}

\def\cW{{\mathcal W}}
\def\cX{{\mathcal X}}
\def\cY{{\mathcal Y}}

\def\ccE{\mathscr{E}}

\def\ccT{\mathscr{T}}

\def\ccS{\mathscr{S}}
\def\ccG{\mathscr{G}}

\def\bt{\bm{t}}
\def\bd{\bm{d}}

\def\bP{\bm{P}}
\def\bcP{\bm{\cP}}
\def\hP{\smash{\hat{P}}}
\def\bhP{\smash{\bm{{\hat{P}}}}}
\def\hcP{\smash{{\hat{\cP}}}}

\DeclareMathOperator{\ex}{ex}

\let\polishlcross=\l
\def\l{\ifmmode\ell\else\polishlcross\fi}

\def\bl{\bigl(}
\def\br{\bigr)}

\def\eps{\varepsilon}

\def\lra{\longrightarrow}

\def\bl{\bigr(}
\def\br{\bigr)}

\newtheoremstyle{note}  {4pt}  {4pt}  {\sl}  {}  {\itshape}  {.}  {.5em}          {}
\theoremstyle{note}

\begin{document}
\title[The three edge theorem]
{On a generalisation of Mantel's theorem to uniformly dense hypergraphs}

\author[Christian Reiher]{Christian Reiher}
\address{Fachbereich Mathematik, Universit\"at Hamburg, Hamburg, Germany}
\email{Christian.Reiher@uni-hamburg.de}

\author[Vojt\v{e}ch R\"{o}dl]{Vojt\v{e}ch R\"{o}dl}
\address{Department of Mathematics and Computer Science, 
Emory University, Atlanta, USA}
\email{rodl@mathcs.emory.edu}
\thanks{The second author was supported by NSF grant DMS 1301698.}

\author[Mathias Schacht]{Mathias Schacht}
\address{Fachbereich Mathematik, Universit\"at Hamburg, Hamburg, Germany}
\email{schacht@math.uni-hamburg.de}


\begin{abstract}
For a $k$-uniform hypergraph $F$ 
let $\ex(n,F)$ be the maximum number of edges of a
$k$-uniform $n$-vertex hypergraph $H$
which contains no copy of~$F$.
Determining or estimating $\ex(n,F)$ is a classical and central problem 
in extremal combinatorics. While for $k=2$ this problem is well understood, 
due to the work of Tur\'an and of Erd\H os and Stone, only very little is known 
for $k$-uniform hypergraphs for $k>2$. We focus on the case when $F$
is a~$k$-uniform hypergraph with three edges on~$k+1$ vertices. Already this 
very innocent (and maybe somewhat particular looking) problem is still 
wide open even for~$k=3$.

We consider a variant of the problem where the large hypergraph~$H$ 
enjoys additional hereditary density conditions. Questions of this type were
suggested by Erd\H os and S\'os about 30 years ago. We show that 
every $k$-uniform hypergraph $H$ with density $>2^{1-k}$ with respect to 
every large collection of $k$-cliques induced by sets of $(k-2)$-tuples
contains a copy of~$F$. The required density $2^{1-k}$ is best possible 
as higher order tournament constructions show.

Our result can be viewed as a common generalisation of the first 
extremal result in graph theory due to Mantel (when $k=2$ and the hereditary density 
condition reduces to a normal density condition) and a recent result of Glebov, Kr{\'a}{\soft{l}}, and Volec
(when $k=3$ and large subsets of vertices of $H$ induce a subhypergraph of density $>1/4$).
Our proof for arbitrary $k\geq 2$ utilises the regularity 
method for hypergraphs.
\end{abstract}

\maketitle

\section{Introduction} 
\label{sec:intro}

\subsection{Tur\'{a}n's hypergraph problem} 
\label{subsec:thp}
A {\it $k$-uniform hypergraph} is a pair
$F=(V, E)$, where~$V$ is a finite set of {\it vertices} and 
$E\subseteq V^{(k)}=\{e\subseteq V\colon |e|=k\}$
is a set of $k$-element subsets of $V$, whose members are called the {\it edges} of $F$.
As usual $2$-uniform hypergraphs are simply called graphs.
With his seminal work~\cite{Tu41}, Tur\'{a}n established a new research area in combinatorics by initiating the systematic study of the 
so-called {\it extremal function} associated with any such hypergraph~$F$. This 
function maps every positive integer~$n$ to the largest number $\ex(n, F)$ of edges 
that an \emph{$F$-free}, $k$-uniform hypergraph~$H$ on $n$ vertices can have, i.e., an $n$-vertex hypergraph 
without containing $F$ as a (not necessarily induced) subhypergraph. 
It is not hard to observe that for every $k$-uniform hypergraph $F$ the limit
\[
	\pi(F)=\lim_{n\to\infty}\frac{\ex(n, F)}{\binom{n}{k}}\,,
\]
known as the {\it Tur\'{a}n density of $F$}, exists. The problem of determining the
Tur\'{a}n densities of all hypergraphs is likewise referred to as {\it Tur\'{a}n's 
hypergraph problem} in the literature.

The first nontrivial instance of these problems is
the case where $k=2$ and $F=K_3$ is a triangle, i.e., the unique 
graph
with three vertices and three edges. More than a century ago, Mantel~\cite{Ma07} 
proved $\ex(n, K_3)=\lfloor n^2/4\rfloor$ for every positive integer $n$. 
Let us record an immediate consequence of this result.  

\begin{thm}[Mantel] \label{thm:Mantel}
We have $\pi(K_3)=\tfrac 12$.
\end{thm}

The next step was taken by Tur\'{a}n himself~\cite{Tu41}, who proved that more generally 
we have~$\pi(K_r)=\tfrac{r-2}{r-1}$ for each integer $r\ge 2$, where $K_r$ denotes the
graph on $r$ vertices with all possible $\binom{r}{2}$ edges. This was further generalised by Erd\H{o}s
and Stone~\cite{ErSt46} and from their result one easily gets the full answer to the 
Tur\'{a}n density problem in the case of graphs. Notably, we have
\[
	\pi(F)=\frac{\chi(F)-2}{\chi(F)-1} 
\]
for every graph $F$ with at least one edge, where $\chi(F)$ denotes the {\it chromatic
number} of~$F$, i.e., the least integer $r$ for which there exists a graph homomorphism
from $F$ to~$K_r$ (see also~\cite{ErSi66}, where the result in this form appeared first). 

Despite these fairly general results about graphs, the current knowledge about 
Tur\'{a}n densities of general hypergraphs is very limited, even in the $3$-uniform case. 
For instance, concerning the $3$-uniform hypergraphs $K_4^{(3)-}$ and $K_4^{(3)}$ 
on four vertices with three and four edges respectively, it is only known that
\[
		\frac{2}{7}\leq\pi(K_4^{(3)-})\leq 0.2871
		\qqand
		\frac{5}{9}\leq\pi(K_4^{(3)})\leq 0.5616\,.
\] 
The lower bounds are due to Frankl and F\"uredi~\cite{FrFu84} and to Tur\'{a}n (see, e.g.,~\cite{Er77}).
In both cases they are believed to be optimal and they are derived from explicit constructions.
The upper bounds were obtained by computer assisted calculations based on Razborov's
\emph{flag algebra method }introduced in~\cite{Ra07}.
They are due to Baber and Talbot~\cite{BaTa11}, and to Razborov himself~\cite{Ra10}. 

This scarcity of results, however, is not due to a lack of interest or effort by 
combinatorialists. It rather seems that these problems are
hard for reasons that might not be completely understood yet. 
For a more detailed discussion we refer to Keevash's survey~\cite{Ke11}.

\subsection{Tur\'{a}n problems in vertex uniform hypergraphs}
\label{subsec:fail}

A variant of these questions suggested by Erd\H{o}s and S\'{o}s 
(see e.g.,~\cites{ErSo82, Er90}) concerns $F$-free hypergraphs $H$
that are uniformly dense with respect to sets of vertices.  

\begin{dfn} \label{dfn:vtxdense}
For real numbers $d\in[0, 1]$ and $\eta>0$ we say that a $k$-uniform hypergraph 
$H=(V, E)$ is {\it $(d, \eta, 1)$-dense} if for all $U\subseteq V$ the 
estimate 
\[
	|U^{(k)}\cap E|\ge d\binom{|U|}{k}-\eta\,|V|^k
\]
holds. 
\end{dfn} 

This means that when one passes to a linearly sized induced subhypergraph of $H$ one 
still has an edge density that cannot be much smaller than $d$. This notion
is closely related to ``vertex uniform'' or ``weakly quasirandom'' hypergraphs
appearing in the literature.
The~$1$ in $(d, \eta, 1)$-dense is supposed to indicate that the density 
is measured with respect to subsets of vertices, which one may view as $1$-uniform
hypergraphs. Our reason for including it is that we intend to put this concept into a broader 
context with Definition~\ref{dfn:jdense} below, where we will allow arbitrary 
$j$-uniform hypergraphs to play the r\^{o}le of the subset $U$. 
For now for every $k$-uniform hypergraph $F$ one may define
\begin{multline*}
	\pi_1(F)=\sup\bigl\{d\in[0,1]\colon \text{for every $\eta>0$ and $n\in \NN$ there exists}\\
	\text{a $k$-uniform, $F$-free,  $(d,\eta, 1)$-dense hypergraph $H$ with $|V(H)|\geq n$}\bigr\}\,.
\end{multline*}

Erd\H{o}s~\cite{Er90} realised that a randomised version of the ``tournament hypergraphs'',
which appeared in joint work with Hajnal~\cite{ErHa72}, provides the lower bound 
\[
	\pi_1\bl K^{(3)-}_4\br\ge \tfrac 14\,,
\] 
and asked (together with S\'os) whether equality holds.
This was recently confirmed by Glebov, Kr{\'a}{\soft{l}}, and Volec~\cite{GKV}
and an alternative proof appeared in~\cite{RRS-a}.

\begin{thm}[Glebov, Kr{\'a}{\soft{l}} \& Volec]
\label{thm:K4-}
We have $\pi_1\bl K^{(3)-}_4\br=\tfrac 14$.
\end{thm} 

Let us recall the construction of the tournament hypergraphs 
yielding the lower bound. 

\begin{exmp}[Tournament construction]
\label{ex:tourn3}
Consider for some positive integer $n$ a random tournament $T_n$ with vertex set 
$[n]=\{1, \ldots, n\}$. This means that we direct each unordered pair 
$\{i, j\}\in [n]^{(2)}$ uniformly at random either as $(i, j)$  or as $(j, i)$. 
These~$\binom{n}{2}$ choices are supposed to be mutually independent. 
In $T_n$ we see on any triple $\{i,j,k\}\in [n]^{(3)}$  either 
a cyclically oriented triangle or a transitive tournament. 
Let $H(T_n)$ be the random $3$-uniform hypergraph with vertex set $[n]$ having exactly 
those triples $\{i,j,k\}$ as edges that span a cyclic triangle in $T_n$. 
Clearly this happens for every fixed triple $\{i,j,k\}$ with probability~$\tfrac 14$ 
and using standard probabilistic arguments it is not hard 
to check that, moreover, for every fixed $\eta>0$ the probability that the 
tournament hypergraph $H(T_n)$ is $\bl \tfrac 14, \eta, 1\br$-dense tends 
to $1$ as $n$ tends to infinity (see Lemma~\ref{lem:Hdense} below). 
Further one sees easily that $H(T_n)$ can never contain a $K^{(3)-}_4$,
since any four vertices can span at most two cyclic triangles in any tournament. 
Consequently, we have indeed $\pi_1\bl K^{(3)-}_4\br\ge \tfrac 14$.
\end{exmp}
 
Theorem~\ref{thm:K4-} asserts that the tournament hypergraph $H(T_n)$
is optimal among uniformly dense
$K^{(3)-}_4$-free hypergraphs. There have been attempts to generalise 
this statement to the class of $k$-uniform hypergraphs
(see e.g., \cite{FrRo88} or~\cite{GKV}).
The work presented here has the goal to formulate and verify such an extension
(see Theorem~\ref{thm:main} below).  For that we generalise the construction from
Example~\ref{ex:tourn3}. We consider hypergraphs arising from 
\emph{higher order tournaments},  which make use of some standard concepts from 
simplicial homology theory (see e.g.,~\cite{Po}*{Ch.~I,~\snake  4}).

\subsection{Higher order tournaments}
\label{subsec:hot}
Let $X$ denote some nonempty finite set. 
By an {\it enumeration} of $X$ we mean
a tuple~$(x_1, \ldots, x_\ell)$ with
$X=\{x_1, \ldots, x_\ell\}$ and $\ell=|X|$. 
An {\it orientation} of $X$ is obtained by putting a factor of $\pm 1$
in front of such an enumeration.
Two orientations $\eps(x_1, \ldots, x_\ell)$ 
and~$\eps'(x_{\tau(1)}, \ldots, x_{\tau(\ell)})$ are {\it identified} 
if either $\eps=\eps'$ and the permutation $\tau$ is even, 
or if $\eps=-\eps'$ and $\tau$ is odd.
So altogether there are exactly two orientations of $X$, 
and we may write, e.g., $+(x, y, z) = -(z, y, x)$. 

If $|X|\ge 2$ and $x\in X$, then every orientation $\sigma$ of $X$ induces an 
orientation $\sigma_x$ of~$X\setminus\{x\}$ in the following way: One picks a representative 
of $\sigma$ having $x$ at the end of the enumeration, and then one removes $x$. 
For instance, the orientation $+(x, y)$ of $\{x, y\}$ induces the orientations 
$+(x)$ and $-(y)$ of $\{x\}$ and $\{y\}$ respectively, while $+(x, y, z)$ induces 
$+(x, y)$, $+(y, z)$, and $-(x, z)$.

For any integers $n\ge k\ge 2$ an {\it $(k-1)$-uniform tournament}  
$T_n^{(k-1)}$ is given by selecting one of the two possible orientations 
of every $(k-1)$-element subset of $[n]$. 
We associate with any $(k-1)$-uniform tournament $T_n^{(k-1)}$
a $k$-uniform hypergraph $H\bl T_n^{(k-1)}\br$ with 
vertex set~$[n]$ by declaring a $k$-element set $e\in [n]^{(k)}$ to be an edge if and only if
there is an orientation on $e$ that induces on all $(k-1)$-element subsets of $e$ the orientation
provided by~$T_n^{(k-1)}$. For $k\geq 3$ this is equivalent to saying that any two distinct 
$x, x'\in e^{(k-1)}$ induce opposite orientations of the $(k-2)$-set $x\cap x'$,
which follows from the fact that the $(k-2)$-nd (simplicial) homology group of 
a $(k-1)$-simplex vanishes.

Moreover, if
$T_n^{(k-1)}$ gets chosen uniformly at random, then for $k\ge 3$ the probability that 
$H\bl T^{(k-1)}_n\br$ is $(2^{1-k}, \eta, 1)$-dense tends for each fixed positive 
real number $\eta$ to $1$ as $n$ tends to infinity.
 
For $k=2$, however, the last mentioned fact is wrong. In fact $H\bl T_n^{(1)}\br$ 
is always a complete bipartite graph and in the random case the sizes of its vertex 
classes are with high probability close to $\tfrac n2$. 
Such graphs may be used to demonstrate the lower bound in Mantel's 
theorem. One is thus prompted to believe that both, this very old 
result and the fairly new Theorem~\ref{thm:K4-}, are special cases of a more general 
theorem about $k$-uniform hypergraphs.

The next step towards finding this common generalisation is to come up with
a \hbox{$k$-uniform} hypergraph that cannot appear in $H\bl T_n^{(k-1)}\br$. 
Notice to this end that up to isomorphism there is just one $k$-uniform hypergraph 
$F^{(k)}$ on $(k+1)$ vertices with three edges,
and that $F^{(2)}=K_3$ while $F^{(3)}=K^{(3)-}_4$.
We observe that the higher order tournament construction always 
gives $F^{(k)}$-free hypergraphs.

\begin{fact} \label{fact:Fk-free} For every $k\geq 2$ 
the hypergraph $H\bl T_n^{(k-1)}\br$ is always $F^{(k)}$-free.
\end{fact}

\begin{proof} For $k=2$ 
the graph $H(T^{(1)})$ is always a complete bipartite graph, which clearly 
contains no triangle. For $k\geq 3$ we 
argue indirectly. Let $w\cup \{a, b\}$, $w\cup\{a, c\}$, and $w\cup\{b, c\}$ 
be the three edges of an $F^{(k)}$ in $H\bl T_n^{(k-1)}\br$, where $|w|=k-2$ 
and $a, b, c\not\in w$ are distinct. 
The three $(k-1)$-sets $w\cup\{a\}$, $w\cup\{b\}$, and $w\cup\{c\}$ receive orientations
from~$T_n^{(k-1)}$, that in turn induce orientations of~$w$. At least two of
these orientations of $w$ must coincide, say the ones induced by $w\cup\{a\}$ and 
$w\cup\{b\}$. But now $w\cup \{a, b\}$ cannot be an edge of~$H\bl T_n^{(k-1)}\br$,
contrary to our assumption. 
\end{proof}

So what we are looking for is a precise sense in which the random higher order tournament 
hypergraph~$H\bl T_n^{(k-1)}\br$ is optimal among $F^{(k)}$-free hypergraphs. 
For $k\ge 3$ in~\cite{FrRo88} and~\cite{GKV} the question whether
\begin{equation}\label{eq:false}
	\pi_{1}\bl F^{(k)}\br\overset{?}{=}2^{1-k}
\end{equation}
was suggested.

For $k=2$ this statement fails, since~$\pi_{1}(F)=0$ 
for every graph~$F$ (see e.g.,~\cite{Ro86}). 
Moreover, a $4$-uniform hypergraph considered in a different context by Leader and 
Tan~\cite{LeTa10} shows that the above formula fails for $k=4$ also.
In fact, this example 
gives the lower 
bound $\pi_1\bl F^{(4)}\br\ge\tfrac 14$ and we will return to this construction in 
Section~\ref{subsec:exmp}.

\subsection{A generalised Tur\'{a}n problem} 
\label{subsec:sym}
All this seems to indicate that vertex uniformity might not be the correct notion of
being ``uniformly dense'' for a common generalisation of Theorem~\ref{thm:Mantel}
and Theorem~\ref{thm:K4-}. The goal of the present subsection is to introduce a stronger 
concept of uniformly dense hypergraphs so that for the corresponding Tur\'an density
a statement like~\eqref{eq:false} becomes true. 
 
Given a $j$-uniform hypergraph $G^{(j)}$ with $j<k$ we denote 
the collection of $k$-subsets of its vertex set that span cliques $K^{(j)}_k$ 
of size $k$ by $\cK_k\bl G^{(j)}\br$.

\begin{dfn}\label{dfn:jdense}
For $d\in [0,1]$, $\eta>0$, and $j\in [0, k-1]$ a $k$-uniform hypergraph $H=(V, E)$
is {\it $(d, \eta, j)$-dense} if 
\[
	\big|\cK_k\bl G^{(j)}\br\cap E\big|\ge d\, \big|\cK_k\bl G^{(j)}\br\big|-\eta\,|V|^k
\]
holds for all $j$-uniform hypergraphs $G^{(j)}$ with vertex set $V$. 
\end{dfn}

In the degenerate case $j=0$ this simplifies to $H$ being $(d, \eta, 0)$-dense if
\begin{equation}\label{eq:j=0}
	|E|\ge d\binom{|V|}{k}-\eta\,|V|^k\,,
\end{equation}
since on any set $V$ there are only two $0$-uniform hypergraphs -- the one with empty edge set
and the one with the empty set being an edge.
Also note that for $j=1$ we recover Definition~\ref{dfn:vtxdense}. We proceed by setting
\begin{multline*}
	\pi_j(F)=\sup\bigl\{d\in[0,1]\colon \text{for every $\eta>0$ and $n\in \NN$ there exists}\\
	\text{a $k$-uniform, $F$-free,  $(d,\eta, j)$-dense hypergraph $H$ with $|V(H)|\geq n$}\bigr\}
\end{multline*}
for every $k$-uniform hypergraph $F$ and propose the following problem.
(A more general problem 
will be discussed in Section~\ref{subsec:ide}.)

\begin{prob} \label{prob:sym}
Determine $\pi_{j}(F)$ for all $k$-uniform hypergraphs $F$ and all $j\in [0, k-2]$.
\end{prob}
Notice that, owing to~\eqref{eq:j=0}, the special case where $j=0$ corresponds to Tur\'{a}n's
classical question of determining $\pi(F)$. At the other end for $j=k-1$ it 
is known that $\pi_{k-1}(F)=0$ for every $k$-uniform hypergraph $F$. This essentially follows from the work in~\cite{KRS02}
(or it can also be verified by means of a straightforward application of the hypergraph regularity method).
Moreover, it seems that for a 
fixed hypergraph $F$ the Problem~\ref{prob:sym} has the tendency of becoming 
easier the larger we make $j$. The reason for this might be that by increasing $j$ 
one gets a stronger hypothesis about the hypergraphs $H$ in which one intends 
to locate a copy of $F$ and this additional information seems to be very helpful. 
In fact, for every $k$-uniform hypergraph $F$ we have
\begin{equation}\label{eq:chain}
	\pi(F)=\pi_{0}(F)\ge \pi_{1}(F)\ge \dots \ge \pi_{k-2}(F)\ge \pi_{k-1}(F)=0\,,
\end{equation}
since $\cK_k(G^{(j)})=\cK_k(G^{(j+1)})$ for every $j$-uniform hypergraph $G^{(j)}$ 
with $G^{(j+1)}=\cK_{j+1}(G^{(j)})$. 
For fixed $F$ the quantities appearing in this chain of inequalities will probably 
be the harder to determine the further they are on the left. For example 
determining~$\pi_j(K^{(k)}_{\l})$ for cliques and $j\leq k-3$ is at least as hard as Tur\'an's original 
problem for $3$-uniform hypergraphs. This suggest that  
Problem~\ref{prob:sym} for the case $j=k-2$ is the first interesting case 
and we will focus on $\pi_{k-2}(\cdot)$ here.

\subsection{The three edge theorem}
\label{subsec:main}

Let us now resume our discussion of higher order tournament hypergraphs 
and of the extremal problem for $F^{(k)}$. Our main result 
is the following.

\begin{thm}[Three edge theorem]
\label{thm:main}
We have $\pi_{k-2}\bl F^{(k)}\br = 2^{1-k}$ for every $k\ge 2$.
\end{thm}

It may be observed that for $k=2$ this gives Mantel's  theorem (Theorem~\ref{thm:Mantel}),
whilst for $k=3$ we get Theorem~\ref{thm:K4-}, meaning that we have indeed found a 
common generalisation of those two results. Below we show that random higher order 
tournament hypergraphs give
the lower bound in Theorem~\ref{thm:main}, which generalises the lower bound constructions of 
Theorems~\ref{thm:Mantel} and~\ref{thm:K4-}.

\begin{lemma} \label{lem:Hdense}
For $k\geq 2$ and $\eta>0$ the probability that $H\bl T_n^{(k-1)}\br$ is 
$(2^{1-k}, \eta, k-2)$-dense when the tournament $T_n^{(k-1)}$ gets chosen 
uniformly at random tends to $1$ as $n$ tends to infinity. 
\end{lemma}

\begin{proof}
Let $E$ denote the random set of edges of $H\bl T_n^{(k-1)}\br$.
Notice that for every $y\in [n]^{(k)}$ the probability of the event ``$y\in E$'' is 
$2^{1-k}$.
This is because~$y$ has $2$ orientations, each of which has a chance of $2^{-k}$ to
match the orientations of the $k$ members of $y^{(k-1)}$. 

Now the key point is that by changing the orientation of one $(k-1)$-subset of~$[n]$ 
we can change $|E|$ by at most $n$. Since
\[
	\eta n^k=n\cdot\binom{n}{k-1}^{1/2}\cdot\Theta\bl n^{(k-1)/2}\br\,,
\]
it follows from the Azuma--Hoeffding 
inequality (see, e.g.,~\cite{JLR00}*{Corollary~2.27})
that for every $(k-2)$-uniform hypergraph $G^{(k-2)}$
with vertex set $[n]$ the bad event that 
\[
	\big|\cK_k\bl G^{(k-2)}\br\cap E\big|< 2^{1-k}\, \big|\cK_k\bl G^{(k-2)}\br\big|-\eta n^k
\] 
happens has at most the probability $e^{-\Omega(n^{k-1})}$. There are only $e^{O(n^{k-2})}$
possibilities for~$G^{(k-2)}$ and, consequently, the union bound tells us that the probability
that $H\bl T_n^{(k-1)}\br$ fails to be~$(2^{1-k}, \eta, k-2)$-dense is at most
$e^{O(n^{k-2})-\Omega(n^{k-1})}=o(1)$. 
\end{proof}

Combining Fact~\ref{fact:Fk-free} and Lemma~\ref{lem:Hdense} yields 
\begin{equation}\label{eq:low}
	\pi_{k-2}\bl F^{(k)}\br \ge 2^{1-k}
\end{equation}
for every $k\geq 2$,
which establishes the lower bound 
of Theorem~\ref{thm:main}.

The upper bound is the main result of this work  
and the proof has some similar features with our alternative proof of  Theorem~\ref{thm:K4-} 
from~\cite{RRS-a}. That proof relies on the regularity method for $3$-uniform hypergraphs, 
so this time we will apply analogous results about $k$-uniform hypergraphs to $H$. 
It appears, however, that a crucial argument from~\cite{RRS-a} occurring after the 
regularisation fails to extend to the general case, even though 
it would not be too hard to adapt it to the case $k=4$. However, for the general case 
new ideas were needed, which are presented in Sections~\ref{sec:proof}--\ref{sec:triangle}.

\subsection*{Organisation} \label{subsec:organ}
In Section~\ref{subsec:ide} we 
introduce further generalised Tur\'an densities, 
and discuss some of their basic properties. The upper bound from Theorem~\ref{thm:main}
will be proved in the Sections~\ref{sec:reduced}--\ref{sec:triangle}.
This begins with revisiting the regularity method for $k$-uniform hypergraphs 
in Section~\ref{sec:reduced}.
What we gain by applying this method is described in Section~\ref{sec:rreduced}. 
Notably it will be shown there that for proving Theorem~\ref{thm:main} it suffices 
to prove a certain statement about ``reduced hypergraphs'' (see 
Proposition~\ref{prop:reduced}). 
In Section~\ref{sec:proof} this task will in turn be reduced to the verification 
of two graph theoretical results, a ``path lemma'' and a ``triangle lemma.''
These will then be proved in Section~\ref{sec:path} and 
Section~\ref{sec:triangle} respectively.   
Finally in Section~\ref{sec:conclude} we will make some concluding remarks concerning
a strengthening of Theorem~\ref{thm:main} for {\it ordered hypergraphs} and
questions for further research.

\section{A further generalisation of Tur\'an's problem}
\label{subsec:ide}

There is a further generalisation of Problem~\ref{prob:sym} 
which will allow us to
replace the assumption of~$H$ being $(2^{1-k}+\eps, \eta, k-2)$-dense in 
Theorem~\ref{thm:main} by the more manageable assumption of~$H$ being
what we call $\bl 2^{1-k}+\eps, \eta, [k]^{(k-2)}\br$-dense (see Proposition~\ref{prop:main} below). 
This condition has the 
advantage of saying something about the edge distribution of $H$ relative to families 
consisting of $\binom{k}{2}$ many $(k-2)$-uniform hypergraphs rather than just relative 
to one such hypergraph.

Given a finite set $V$ and a set $S\subseteq [k]$ we write $V^S$ for the set of all 
functions from~$S$ to~$V$. 
It will be convenient to identify the Cartesian power 
$V^k$ with $V^{[k]}$ by regarding any $k$-tuple $\seq{v}=(v_1, \ldots, v_k)$ as
being the function $i\longmapsto v_i$.
In this way, the natural projection from $V^k$ to $V^S$ becomes the restriction 
$\seq{v}\longmapsto \seq{v}\,|\,S$ and the preimage of any set $G_S\subseteq V^S$ is denoted by
\[
	\cK_k(G_S)=\bigl\{\seq{v}\in V^k\colon (\seq{v}\,|\,S)\in G_S\bigr\}\,.
\]
One may think of $G_S\subseteq V^S$ as a directed hypergraph (where vertices in the directed hyperedges are also allowed to repeat).

More generally, when we have a subset $\ccS\subseteq\powerset([k])$ of the power set of $[k]$
and a family $\ccG=\{G_S\colon S\in \ccS\}$ with $G_S\subseteq V^S$ for all $S\in\ccS$,
then we will write
\begin{equation}\label{eq:711}
	\cK_k(\ccG)=\bigcap_{S\in\ccS}\cK_k(G_S)\,.
\end{equation}
If moreover $H=(V, E)$ is a $k$-uniform hypergraph on $V$, then $e_H(\ccG)$
denotes the cardinality of the set
\[
	E_H(\ccG)=\bigl\{(v_1, \ldots, v_k)\in \cK_k(\ccG)\colon\{v_1, \ldots, v_k\}\in E\bigr\}\,.
\]
Now we are ready to state our main definitions. 

\begin{dfn} \label{dfn:dense}
Let real numbers $d\in [0, 1]$ and $\eta>0$, a $k$-uniform hypergraph $H=(V, E)$
and a set $\ccS\subseteq\powerset([k])$ be given. We say that $H$ is {\it $(d, \eta, \ccS)$-dense}
provided that 
\[
	e_H(\ccG)\ge d\,|\cK_k(\ccG)|-\eta\,|V|^k
\]
holds for every family $\ccG=\{G_S\colon S\in \ccS\}$ associating with each $S\in\ccS$ 
some $G_S\subseteq V^S$. 
\end{dfn}  

For example, if $k=3$ and $\ccS=\EV=\bigl\{\{1, 2\}, \{3\}\bigr\}$, then it is convenient to identify the sets
$V^{\{1,2\}}\cong V\times V$ and $V^{\{3\}}\cong V$. 
This way saying that $H$ is $(d, \eta, \EV)$-dense means 
that for all sets $G_{\{1,2\}}\subseteq V\times V$ 
and $G_{\{3\}}\subseteq V$ there are at least 
\[
	d\,|G_{\{1,2\}}|\,|G_{\{3\}}|-\eta\,|V|^3
\] 
triples $(x, y, z)\in V^3$
such that $(x, y)\in G_{\{1,2\}}$, $z\in G_{\{3\}}$, and $\{x, y, z\}\in E$. The reader may consult~\cite{RRS-b} 
for a systematic discussion of essentially all such density notions arising for $k=3$.  

Definition~\ref{dfn:dense} leads us in the expected way to further generalised 
Tur\'{a}n densities.

\begin{dfn}\label{def:gpi}
Given a $k$-uniform hypergraph $F$ and a set $\ccS\subseteq\powerset([k])$ we put
\begin{multline*}
	\pi_{\ccS}(F)=\sup\bigl\{d\in[0,1]\colon \text{for every $\eta>0$ and $n\in \NN$ 
	there exists}\\
	\text{a $k$-uniform, $F$-free,  $(d,\eta, \ccS)$-dense hypergraph $H$ with 
	$|V(H)|\geq n$}\bigr\}\,.
\end{multline*}
\end{dfn}
As we shall see in Proposition~\ref{prop:technical} below, for symmetrical 
families $\ccS$ of the form $[k]^{(j)}$ the functions 
$\pi_{j}(\cdot)$ and $\pi_{[k]^{(j)}}(\cdot)$ coincide.
Consequently, the following problem generalises Problem~\ref{prob:sym}.

\begin{prob} \label{prob:general}
Determine $\pi_{\ccS}(F)$ for all $k$-uniform hypergraphs $F$ and all 
${\ccS\subseteq\powerset([k])}$.
\end{prob}

However, Problem~\ref{prob:general} does also make sense when $\ccS$ is not 
``symmetrical'' and it seems to us that these most general Tur\'an 
densities have interesting properties. For instance, the inequality 
\[
	 \piee(K_{2^r}^{(3)}) \leq \frac{r-2}{r-1} \leq \piev(K_{r+1}^{(3)})
	 \quad \text{for } r\ge 2\,,
\]
where $\EV=\bigl\{\{1, 2\}, \{3\}\bigr\}$ and $\EE=\bigl\{\{1, 2\}, \{2, 3\}\bigr\}$, 
shows that there is a striking discrepancy between the growth rates of~$\piee(\cdot)$ 
and $\piev(\cdot)$ for $3$-uniform cliques (see~\cite{RRS-c}).

Let us record some easy monotonicity 
properties of these generalised Tur\'{a}n densities, which generalise~\eqref{eq:chain}
and show that it suffices to study $\pi_\ccS(F)$ when $\ccS$
in an antichain.

\begin{prop} \label{prop:piS}
Let $F$ denote some $k$-uniform hypergraph.
\begin{enumerate}[label=\alabel]
\item\label{it:piSa} 
If $\ccS\subseteq\ccT\subseteq\powerset([k])$, then $\pi_{\ccS}(F)\ge \pi_{\ccT}(F)$.
\item\label{it:piSb}
If $\ccT\subseteq\powerset([k])$ and $\ccS\subseteq\ccT$ is the set of those members of $\ccT$ 
that are maximal with respect to inclusion, then $\pi_{\ccS}(F)= \pi_{\ccT}(F)$.
\end{enumerate}
\end{prop}
\begin{proof}
Part~\ref{it:piSa} follows from the fact that every $(d, \eta, \ccT)$-dense hypergraph
is a fortiori $(d, \eta, \ccS)$-dense. 

For the proof of part~\ref{it:piSb} we note that in view of~\ref{it:piSa} it suffices to show $\pi_{\ccS}(F)\leq \pi_{\ccT}(F)$. Owing  
Definition~\ref{def:gpi} it suffices to show that every 
$(d, \eta, \ccS)$-dense hypergraph $H=(V, E)$ is also $(d, \eta, \ccT)$-dense. 
Proceeding by induction on $|\ccT\setminus\ccS|$ this claim gets reduced to the special 
case where $\ccT=\ccS\cup\{A\}$ and $A\subseteq B\in \ccS$ hold for some sets~$A$ and $B$.
Now let $\ccG_{\ccT}=\{G_T\colon T\in \ccT\}$ be any family with $G_T\subseteq V^T$ 
for all~$T\in\ccT$. The set 
\[
	G'_B=\bigl\{ \seq{v}\in G_B\colon(\seq{v}\,|\,A)\in G_A\bigr\}\subseteq V^B
\]
has the property that $\cK_k(G'_B)=\cK_k(\{G_A, G_B\})$. Thus if we set $G'_S=G_S$ for all 
$S\in \ccS\setminus\{A\}$ and $\ccG_{\ccS}=\{G'_S\colon S\in\ccS\}$, then 
$\cK_k(\ccG_{\ccS})=\cK_k(\ccG_\ccT)$ and, hence, $e_H(\ccG_{\ccT})=e_H(\ccG_{\ccS})$ and 
the $(d, \eta, \ccT)$-denseness of $H$ follows from its $(d, \eta, \ccS)$-denseness.
\end{proof}

We conclude this section with 
the following observation, which for $j=k-2$ will be useful in the proof of Theorem~\ref{thm:main}.

\begin{prop}\label{prop:technical}
If $F$ is a $k$-uniform hypergraph and $j\in [1,k-1]$, then $\pi_{j}(F)=\pi_{[k]^{(j)}}(F)$.
\end{prop} 
The curious reader may wonder what happens for the case $j=0$ and, in fact, the proposition
also holds in this somewhat peculiar case. 
\begin{proof}
	First we observe that $\pi_{j}(F)\geq \pi_{[k]^{(j)}}(F)$ for all $k>j\geq0$ and 
	every $k$-uniform hypergraph~$F$. This follows from the observation that 
	every $(d,\eta,[k]^{(j)})$-dense 
	$k$-uniform hypergraph $H=(V,E)$ is also $(d,\eta/k!,j)$-dense. 
	
	Indeed, to see this we
	consider a 
	$j$-uniform hypergraph~$G^{(j)}$ with vertex set~$V$. We shall apply the 
	$(d,\eta,[k]^{(j)})$-denseness of $H$
	to the family~$\ccG$ consisting for every $J\in[k]^{(j)}$ 
	of a ``directed'' copy $G_J$ of $G^{(j)}$, i.e., $(v_i)_{i\in J}\in G_J$ if $\{v_i\colon i\in J\}\in G^{(j)}$.
	Recall, that~$\cK_k(G^{(j)})$ contains all $k$-element subsets of $V$ that span a clique 
	in $G^{(j)}$. On the other hand, $\cK_k(\ccG)$ contains every ordered $k$-tuple 
	$\seq{v}=(v_1,\dots,v_k)\in V^k$ such that for every $J\in[k]^{(j)}$ the projection $(\seq{v}\,|\,J)$
	is in $G_J$, which by definition means $\{v_i\colon i\in J\}\in G^{(j)}$. 
	Consequently, every $k$-element set from $\cK_k(G^{(j)})$ appears in all $k!$ orderings in 
	$\cK_k(\ccG)$ and every ordered $k$-tuple from $\cK_k(\ccG)$ appears unordered in $\cK_k(G^{(j)})$.
	This yields 
	\[
		|\cK_k(G^{(j)})|=\frac{|\cK_k(\ccG)|}{k!}
	\]
	and, similarly, we have 
	\[
		|\cK_k(G^{(j)})\cap E|=\frac{e_H(\ccG)}{k!}\,,
	\]
	which implies 
	\[
 		|\cK_k(G^{(j)})\cap E|
		=
		\frac{e_H(\ccG)}{k!}
		\geq
		\frac{1}{k!}(d\,|\cK_k(\ccG)|-\eta n^k)
		=
		d\,|\cK_k(G^{(j)})|-\frac{\eta}{k!} n^k
	\]
	and the observation follows.

	For the opposite inequality
	\begin{equation}\label{eq:ptech}
		\pi_{j}(F)\leq \pi_{[k]^{(j)}}(F)
	\end{equation}
	we distinguish the cases $j=1$ and $j\geq 2$.

	Perhaps somewhat surprisingly the proof for $j\geq 2$ seems to be simpler 
	than the case $j=1$ and we give it first. In fact, one can again prove (as above) that 
	Definitions~\ref{dfn:jdense} and~\ref{dfn:dense} coincide up to a 
	different value of~$\eta$. More precisely, sufficiently for large $n=|V|$ we show:
	\begin{enumerate}[label=\rmlabel]
		\item\label{it:ptechi} If $j\geq 2$ and $H=(V,E)$ is a $(d,\eta,j)$-dense $k$-uniform hypergraph, 
			then it is 
			also $(d,\eta',[k]^{(j)})$-dense for $\eta'=2k^k\eta$.
	\end{enumerate} 
	For $j=1$ we have to pass to an induced subhypergraph to show a similar assertion.
	\begin{enumerate}[label=\rmlabel]
		\setcounter{enumi}{1}
		\item\label{it:ptechii} Let $j=1$. 
			For every $\eta'>0$ there exists $\eta>0$ so that for every $d>0$ and every
			sufficiently large $(d,\eta,1)$-dense $k$-uniform hypergraph $H=(V,E)$ there 
			exists a subset $U\subseteq V$ of size at least $\eta |V|$ so that the 
			induced subhypergraph $H[U]$ is $(d,\eta',[k]^{(1)})$-dense.
	\end{enumerate}
	\noindent
	\textit{Proof of~\ref{it:ptechi}}.
	We assume by contradiction that there is a system of oriented 
	$j$-uniform hypergraphs $\ccG=\{G_J\colon J\in[k]^{(j)}\}$ with $G_J\subseteq V^J$ such that 
	\begin{equation}\label{eq:ptech-contra}
		e_H(\ccG) <  d\,|\cK_k(\ccG)|-\eta' |V|^k\,.
	\end{equation}
	We consider a random partition $\cP$ of $V=V_1\dcup\dots\dcup V_k$, where each vertex $v\in V$ is included in 
	any $V_i$ independently with probability $1/k$ and set
	\[
		\cK_k^\cP(\ccG)=(V_1\times \dots\times V_k)\cap\cK_k(\ccG)
		\qqand
		E_H^\cP(\ccG)=(V_1\times \dots\times V_k)\cap E_H(\ccG)\,.
	\]
	Using sharp concentration inequalities one can show that with probability tending to~$1$ 
	(as~$n=|V|\to\infty$) we have
	\[
		\big|\cK_k^\cP(\ccG)\big|=(1+o(1))\frac{1}{k^k}\big|\cK_k(\ccG)\big|
		\qqand
		\big|E_H^\cP(\ccG)\big|=(1+o(1))\frac{1}{k^k}e_H(\ccG)\,.
	\]
	Thus, in view of~\eqref{eq:ptech-contra} we infer that there is a partition $\cP$
	such that
	\begin{equation}\label{eq:ptech-contra2}
		\big|E_H^\cP(\ccG)\big| <  d\,\big|\cK_k^\cP(\ccG)\big|-\frac{\eta'}{2k^k} |V|^k
		=d\,\big|\cK_k^\cP(\ccG)\big|-\eta |V|^k\,.
	\end{equation}
	Note that $\cK_k^\cP(\ccG)$ consists of all $k$-tuples 
	$\seq{v}=(v_1,\dots,v_k)\in V_1\times\dots\times V_k\subseteq V^k$ such that 
	\[
		(\seq{v}\,|\,J)\in G_{J}\,.
	\]
	Now we define the $j$-uniform (undirected) hypergraph $G^{(j)}$ on $V$ with edge set
	\[
		\bigdcup_{J\in[k]^{(j)}}\Big\{\{v_i\colon i\in J\}\colon 
		(v_i)_{i\in J}\in G_{J}\cap\prod_{i\in J}V_i \Big\}\,.
	\]
	Since $j\geq 2$, there is a one-to-one correspondence between the $k$-element sets in
	$\cK_k(G^{(j)})$ and the (ordered) $k$-tuples in  $\cK_k^{\cP}(\ccG)$. In fact, we have
	\begin{equation}\label{eq:ptech-clique}
		\cK_k(G^{(j)})=\big\{\{v_1,\dots,v_k\}\colon (v_1,\dots,v_k)\in\cK_k^{\cP}(\ccG)\big\}\,.
	\end{equation}
	(Note that this identity does not hold for $j=1$, since in that case $G^{(1)}$
	and in $\cK_k(G^{(1)})$ additionally those cliques arise which have more than one vertex in some 
	of the vertex classes $V_i$.)
	In view of~\eqref{eq:ptech-contra2} we infer from~\eqref{eq:ptech-clique} for $j\geq 2$ that
	the hypergraph $H$ is not $(d,\eta,j)$-dense for, which concludes the proof of assertion~\ref{it:ptechi}.

	\noindent
	\textit{Proof of~\ref{it:ptechii}}.
	The proof of assertion~\ref{it:ptechii} (for the case $j=1$) relies on a somewhat standard application
	of the so-called \emph{weak 
	hypergraph 	regularity lemma}, which is the straightforward extension of Szemer\'edi's 
	regularity lemma~\cite{Sz} from graphs to hypergraphs. We sketch this proof below.
	
	Given $\eta'>0$ we shall apply the weak hypergraph regularity lemma with $\eps>0$
	and a lower bound on the number of vertex classes $t_0$ and fix an auxiliary constant $\l$ 
	such that
	\[
		1/k,\eta'\gg 1/\l \gg \eps \gg 1/t_0\,.
	\]
	The weak hypergraph regularity lemma yields an upper bound $T_0=T_0(\eps,t_0)$
	on the number of vertex classes in the regular partition and we take
	\[
		\eta\ll 1/T_0\,.
	\]
	Let $n$ be sufficiently large and let $H=(V,E)$ be a $(d,\eta,1)$-dense $k$-uniform hypergraph 
	on $n=|V|$ vertices. The weak hypergraph regularity lemma applied to $H$ yields a partition 
	$V_1\dcup\dots\dcup V_t=V$ with $t_0\leq t\leq T_0$ such that all but at most $\eps t^k$ of the 
	$k$-tuples $K=\{i_1,\dots,i_k\}\in[t]^{(k)}$ the family $(V_{i})_{i\in K}$ are $\eps$-regular, i.e., they satisfy
	\begin{equation}\label{eq:wreg}
		e_H(W_{i_1},\dots,W_{i_k})
		=
		d_K\,|W_{i_1}|\cdot\ldots\cdot|W_{i_k}|\pm\eps|V_{i_1}|\cdot\ldots\cdot|V_{i_k}|
	\end{equation}
	for all $W_{i_1}\subseteq V_{i_1},\dots,W_{i_k}\subseteq V_{i_k}$, where $d_K$
	denotes the density of induced $k$-partite subhypergraph of $H$ on $V_{i_1}\dcup \dots\dcup V_{i_k}$.
	
	As in many proofs utilising the 
	\emph{regularity method} we successively apply Tur\'an and Ramsey-type arguments to obtain a subset 
	$L\subseteq [t]$ of size at least $|L|=\l$ such that for every $k$-element subset 
	$K\in L^{(k)}$ the associated $k$-tuple $(V_i)_{i\in K}$ is $\eps$-regular with density $d_K$ such that 
	\begin{equation}\label{eq:ptech-Ramsey}
		\text{either $d_K\geq d-\eta'/2$ for every $K\in L^{(k)}$ or $d_K<d-\eta'/2$ for every~$K\in L^{(k)}$.}
	\end{equation}
	We shall rule out the latter case by appealing to the $(d,\eta,1)$-denseness of~$H$.
	Applying the $(d,\eta,1)$-denseness to $U=\bigcup_{i\in L} V_i$ yields
	\begin{equation}\label{eq:ptech-lb}
		e_H(U)\geq d\binom{|U|}{k}-\eta n^k\,.
	\end{equation}
	Since at most 
	\begin{equation}\label{eq:712}
		\l\binom{n/t}{2}|U|^{k-2}<\frac{n}{t}|U|^{k-1}=\l^{k-1}\left(\frac{n}{t}\right)^{k}
	\end{equation}
	of the edges of $H[U]$ can 
	intersect some vertex class $V_i$ in more than one vertex, 
	there must be some $K_0\in L^{(k)}$ such that
	\[
		d_{K_0} \geq \frac{e_H(U)-\l^{k-1}(n/t)^{k}}{\binom{\l}{k}(n/t)^{k}}
		\overset{\eqref{eq:ptech-lb}}{\geq} 
		\frac{d\binom{|U|}{k}-\eta n^k-\l^{k-1}(n/t)^{k}}{\binom{\l}{k}(n/t)^{k}}
		\geq
		d-\frac{\eta n^k}{\binom{\l}{k}(n/t)^{k}}-\frac{k^k}{\l}\,.
	\]
	Moreover, since 
	$1/\eta'\ll t_0\leq  t\leq T_0\ll 1/\eta$ we have $\frac{\eta n^k}{\binom{\l}{k}(n/t)^k}\leq \eta'/4$ and 
	from $\l\gg1/\eta'$ we infer $k^k/\l\leq \eta'/4$.
	Consequently, 
	\[
		d_{K_0}\geq d-\frac{\eta'}{2}
	\] 
	and, hence, it follows from~\eqref{eq:ptech-Ramsey} that 
	$d_K\geq d-\eta'/2$ for every $K\in L^{(k)}$.
	
	Since by our choice of constants we also have
	\[
		|U|\geq \frac{\l}{T_0}n\gg \eta n
	\]
	we conclude the proof by showing that the induced hypergraph $H[U]$ is $(d,\eta',[k]^{(1)})$-dense. 
	Roughly speaking, this will be inherited from the $\eps$-regularity of the 
	families~$(V_i)_{i\in K}$ with density $d_K\geq d-\eta'/2$ for all $K\in L^{(k)}$. 
	
	More formally, let $\ccG=\{U_1,\dots,U_k\}$ be an arbitrary family of subsets of~$U$ (which take the r\^ole of the 
	hypergraphs $G_{\{i\}}\subseteq V^{\{i\}}$ in the definition of $(d,\eta',[k]^{(1)})$-denseness).
	Note that in view of~\eqref{eq:711} we have $\cK_k(\ccG)\subseteq V^k$ and
	\[
		|\cK_k(\ccG)|
		=
		|U_1|\cdot\ldots\cdot|U_k|
		=
		\sum_{(i_1,\dots,i_k)\in L^k}|U_1\cap V_{i_1}|\cdot\ldots\cdot|U_k\cap V_{i_k}|\,.
	\]
	Moving to (unordered) $k$-element subsets of $L$ we obtain by similar calculations as in~\eqref{eq:712}
	that
	\[
		|\cK_k(\ccG)| 
		\leq 
		\sum_{K\in L^{(k)}}\sum_{\tau}|U_{\tau(1)}\cap V_{i_1}|\cdot\ldots\cdot|U_{\tau(k)}\cap V_{i_k}|
		+k!\l^{k-1}(n/t)^k\,,
	\]
	where the inner sum runs over all permutations $\tau$ on $[k]$.
	Applying the $\eps$-regularity (see~\eqref{eq:wreg}) to every family 
	$(U_{\tau(1)}\cap V_{i_1},\dots,U_{\tau(k)}\cap V_{i_k})$ yields
	\[
		e_H(\ccG)
		\geq 
		\left(d-\frac{\eta'}{2}\right)\left(|\cK_k(\ccG)|-k!\l^{k-1}\left(\frac{n}{t}\right)^k\right)
		-\binom{\l}{k}k!\cdot\eps\left(\frac{n}{t}\right)^k
		\geq d\,|\cK_k(\ccG)| -\eta' |U|^k\,,
	\]
	since our choice of constants ensures $\eta'\gg1/\l\gg\eps$. 
	Finally, since $\ccG=\{U_1,\dots, U_k\}$ was an arbitrary family  of subsets, 
	this shows that $H[U]$ is $(d,\eta',[k]^{(1)})$-dense as claimed.
\end{proof}
In view of Proposition~\ref{prop:technical} (applied with $j=k-2$ and $F=F^{(k)}$)
combined with~\eqref{eq:low}, the proof of Theorem~\ref{thm:main} reduces to the 
following proposition.
\begin{prop}\label{prop:main}
	For every $k\geq 2$ we have $\pi_{[k]^{(k-2)}}(F^{(k)})\leq 2^{1-k}$,
	i.e., for every $\eps>0$ there exist $\eta>0$ and $n_0\in\NN$
	such that every $\bigl( 2^{1-k}+\eps, \eta, [k]^{(k-2)}\bigr)$-dense
	$k$-uniform hypergraph~$H$ on $n\geq n_0$ vertices 
	contains three edges on $k+1$ vertices.
\end{prop}
The proof of Proposition~\ref{prop:main} is based on the regularity method for hypergraphs, which we introduce in Section~\ref{sec:reduced}.

\section{The hypergraph regularity method}
\label{sec:reduced}
Similar as in the related work~\cites{RRS-a,RRS-b,RRS-c,RRS-d} the regularity method for hypergraphs 
(developed for $k$-uniform hypergraphs in~\cites{Go07,NRS06,RS04}) plays a central r\^ole in the proof of Proposition~\ref{prop:main}. For the intended application we shall utilise the \emph{hypergraph regularity lemma}
and an accompanying \emph{counting/embedding lemma} (see Theorems~\ref{thm:RL} and~\ref{thm:EL} below).
 We follow the approach from~\cites{RSch07RL,RSch07CL} and introduce the necessary notation below.

\subsection{Regular complexes} For a $(j-1)$-uniform hypergraph $P^{(j-1)}$ and a $j$-uniform hypergraph 
$P^{(j)}$ we define the 
\emph{relative density $d(P^{(j)}\,|\,P^{(j-1)})$} of~$P^{(j)}$ 
with respect to~$P^{(j-1)}$ by
\[
	d\big(P^{(j)}\,|\,P^{(j-1)}\big)
	=
	\frac{\big|\cK_{j}(P^{(j-1)})\cap E(P^{(j)})\big|}{\big|\cK_{j}(P^{(j-1)})\big|}
\]
and for definiteness we set $d(P^{(j)}\,|\,P^{(j-1)})=0$ in case $\cK_{j}(P^{(j-1)})=\emptyset$.

As usual we say a bipartite graph $P^{(2)}$ with vertex partition $V_1\dcup V_2$ is
$(\delta,d_2)$-regular, if for all subsets $U_1\subseteq V_1$ and $U_2\subseteq V_2$
we have
\[
	\big|e(U_1,U_2)-d_2|U_1||U_2|\big|
	\leq
	\delta |V_1||V_2|\,.
\]
This definition is extended for $j$-uniform hypergraphs for $j\geq 3$ as follows.
For $j\geq 3$, $d_j\geq 0$ and $\delta>0$ we say a $j$-partite $j$-uniform 
hypergraph $P^{(j)}$ is \emph{$(\delta,d_j)$-regular} w.r.t.\ a $j$-partite $(j-1)$-uniform hypergraph $P^{(j-1)}$ on the same vertex partition, if for 
every subhypergraph $Q\subseteq P^{(j-1)}$ we have
\begin{equation}\label{eq:hreg1}
	\Big|\big|E(P^{(j)})\cap \cK_{j}(Q)\big|-d_j\big|\cK_j(Q)\big|\Big|
	\leq 
	\delta\big|\cK_j(P^{(j-1)})\big|\,.
\end{equation}
In other words, $P^{(j)}$ is regular w.r.t.~$P^{(j-1)}$ if the relative densities 
$d(P^{(j)}\,|\,Q)$ are all approximately the same for all subhypergraphs 
$Q\subseteq P^{(j-1)}$ spanning many $j$-cliques.

Moreover, if $P^{(j-1)}$ and $P^{(j)}$ are $\l$-partite on the same vertex partition $V_1\dcup\dots\dcup V_{\l}$
then we say $P^{(j)}$ is $(\delta,d_j)$-regular w.r.t.~$P^{(j-1)}$ if
$P^{(j)}[V_{i_1},\dots,V_{i_j}]$ is $(\delta,d_j)$-regular w.r.t.\ $P^{(j-1)}[V_{i_1},\dots,V_{i_j}]$
for all $\binom{\l}{j}$ naturally induced $j$-partite subhypergraphs. We shall consider families
$(P^{(2)},\dots,P^{(k-1)})$ of hypergraphs of uniformities $j=2,\dots,k-1$
with $P^{(j)}$ being regular w.r.t.~$P^{(j-1)}$, which leads to the concept of a regular 
complex.

\begin{dfn}[regular complex]
	\label{def:complex}
	We say a family of hypergraphs $\bP=(P^{(2)},\dots,P^{(k-1)})$ is a \emph{$(k-1,\l)$-complex} 
	with vertex partition $V_1\dcup\dots\dcup V_{\l}$ if 
	\begin{enumerate}[label=\rmlabel]
	\item $P^{(j)}$ is an $\l$-partite $j$-uniform hypergraph with vertex partition $V_1\dcup\dots\dcup V_{\l}$
		for every $j=2,\dots,k-1$ and 
	\item $P^{(j)}\subseteq \cK_j(P^{(j-1)})$ for every $j=3,\dots,k-1$.
	\end{enumerate}
	Such a  complex is \emph{$(\delta,\bd)$-regular} for  $\delta>0$ and 
	$\bd=(d_2,\dots,d_{k-1})\in \RR^{k-2}_{\geq 0}$, if in addition
	\begin{enumerate}[label=\rmlabel]
	\setcounter{enumi}{2}
	\item $P^{(2)}$ is $(\delta,d_2)$-regular and 
		$P^{(j)}$ is $(\delta,d_j)$-regular w.r.t.\ $P^{(j-1)}$ for $j=3,\dots,k-1$.
	\end{enumerate}
\end{dfn}

Similarly, as Szemer\'edi's regularity lemma breaks the vertex set of a large graph 
into classes such that most of the bipartite subgraphs induced between the classes are $\eps$-regular,
the regularity lemma for $k$-uniform hypergraphs breaks $V^{(k-1)}$  for a hypergraph $H=(V,E)$
into $(k-1,k-1)$-complexes which are regular themselves and $H$ will be regular on 
most ``naturally induced''  $(k-1,k)$-complexes from that partition. We now describe the 
structure of this underlying auxiliary partition in more detail.

\subsection{Equitable partitions}
The regularity lemma for $k$-uniform hypergraphs provides a
well-structured family of partitions $\bcP=(\cP^{(1)},\dots,\cP^{(k-1)})$
of vertices, pairs, $\dots$, and $(k-1)$-tuples of the vertex set.
We now discuss the structure of these partitions inductively.  Here
the partition classes of $\cP^{(j)}$ will be $j$-uniform
$j$-partite hypergraphs.

Let $V_1\dcup\dots\dcup V_{t_1}=V$ be a partition of some vertex set~$V$
and set $\cP^{(1)}=\{V_1,\dots,V_{t_1}\}$. 
For any $1\leq j\leq t_1$ we consider the $j$-sets $J\in V^{(j)}$
with $|J\cap V_i|\leq 1$ for every $V_i\in\cP^{(1)}$ and due to its similarity to~\eqref{eq:711}
we denote the set of these $j$-sets by~$\cK_j(\cP^{(1)})$, i.e., 
\[
	\cK_j(\cP^{(1)})=\big\{J\in V^{(j)}\colon |J\cap V_i|\leq 1\ \text{for all}\ V_i\in\cP^{(1)}\big\}\,.
\]
Suppose for each $1\leq
i\leq j-1$ partitions $\cP^{(i)}$ of $\cK_i(\cP^{(1)})$ into 
$i$-uniform $i$-partite hypergraphs are given. Then for
every $(j-1)$-set~$J'\in\cK_{j-1}(\cP^{(1)})$, there exists 
a unique $(j-1)$-uniform $(j-1)$-partite hypergraph $P^{(j-1)}_{J'}\in\cP^{(j-1)}$
with $J'\in E(P_{J'}^{(j-1)})$. 
Moreover, for every $J\in\cK_j(\cP^{(1)})$
we define the \emph{polyad} of $J$ by
\[
	\hP_J^{(j-1)}=\bigcup\big\{P^{(j-1)}_{J'}\colon  J'\in J^{(j-1)}\big\}\,.
\]
In other words, $\hP_J^{(j-1)}$ is the unique set of~$j$
partition classes of $\cP^{(j-1)}$ each containing precisely one $(j-1)$-element subset of $J$.
We view $\hP_J^{(j-1)}$ as $j$-partite $(j-1)$-uniform hypergraph with vertex classes $V_i\in\cP^{(1)}$ such that $|V_i\cap J|=1$ and edge set $\bigcup_{J'\in J^{(j-1)}}E(P^{(j-1)}_{J'})$. In general, we shall use the hat-accent 
`$\,\hat\ \,$' for hypergraphs arising from the partition which have more vertex classes than their uniformity
requires. By definition 
we have
\[
	J\in\cK_j(\hP_J^{(j-1)})\,.
\]
More generally, for every $i$ with $1\leq i< j$, we set
\begin{equation}\label{eq:polyad2}
  \hP^{(i)}_J=\bigcup\big\{\cP^{(i)}_I\colon I\in J^{(i)}\big\}\,.
\end{equation}
This allows us for every $J\in\cK_j(\cP^{(1)})$ to consider the 
$(j-1,j)$-complex (see Definition~\ref{def:complex})
\begin{equation}\label{eq:polyad-complex}
	\bhP_J^{(j-1)}=\big(\hP^{(2)}_J,\dots,\hP^{(j-1)}_J\big)\,,
\end{equation}
which ``supports''~$J$.
Consider the family of all polyads
\[
	\hcP^{(j-1)}=\big\{\hP_J^{(j-1)}\colon J\in\cK_j(\cP^{(1)})\big\}\,.
\]
and observe that 
$\{\cK_j(\hP^{(j-1)})\colon \hP^{(j-1)}\in\hcP^{(j-1)}\}$ is a
partition of $\cK_j(\cP^{(1)})$. The structural requirement  
on the partition $\cP^{(j)}$ of $\cK_j(\cP^{(1)})$ is that
\begin{equation}\label{eq:Pj}
  \cP^{(j)}\prec\{\cK_j(\hP^{(j-1)})\colon \hP^{(j-1)}\in\hcP^{(j-1)}\}\,,
\end{equation}
where $\prec$ denotes the refinement relation of set partitions.
This way 
we require that the set of cliques spanned by any polyad in
$\hcP^{(j-1)}$ is subpartitioned in~$\cP^{(j)}$ and every 
partition class of~$\cP^{(j)}$ belongs to precisely one polyad
in~$\hcP^{(j-1)}$, i.e., for every $j$-uniform $j$-partite hypergraph 
$P^{(j)}\in \cP^{(j)}$ there is a unique polyad $\hP^{(j-1)}\in\hcP^{(j-1)}$ with 
$P^{(j)}\subseteq \cK_j(\hP^{(j-1)})$.
Also~\eqref{eq:Pj} implies (inductively)  
that~$\bhP_J^{(j-1)}$ defined in~\eqref{eq:polyad-complex} is indeed a 
$(j,j-1)$-complex.

The hypergraph regularity lemma also provides such a family of partitions with 
the additional property that the number of hypergraphs that partition the cliques of 
a given polyad is independent of the polyad. This leads to the following notion of 
a family of partitions.
\begin{dfn}[family of partitions]
	\label{def:partition}
	Suppose $V$ is a set of vertices and~$\bt=(t_1,\dots,t_{k-1})$ is a vector of positive integers. 
  	We say $\bcP=\bcP(k-1,\bt)=(\cP^{(1)},\dots,\cP^{(k-1)})$
  	is a \emph{family of partitions on~$V$} if 
  \begin{enumerate}[label=\rmlabel]
  \item $\cP^{(1)}$ is a partition $V_1\dcup\dots\dcup V_{t_1}=V$ with $t_1$ classes and
  \item for $j=2,\dots,k-1$  we have that $\cP^{(j)}$ is a partition of~$\cK_j(\cP^{(1)})$ 
  	satisfying~\eqref{eq:Pj}
    and  
    \begin{equation}\label{eq:equit2}
    	\big|\big\{P^{(j)}\in\cP^{(j)}\colon P^{(j)}\subseteq \cK_j(\hP^{(j-1)})\big\}\big|=t_j
    \end{equation}
    for every $\hP^{(j-1)}\in\hcP^{(j-1)}$. 
  \end{enumerate}
  Moreover, we say $\bcP=\bcP(k-1,\bt)$ is $T_0$-bounded, if
  $\max\{t_1,\dots,t_{k-1}\}\leq T_0$.
\end{dfn}

In addition to these structural properties the hypergraph regularity lemma 
provides a family of partitions such that all the complexes ``build by blocks of the partition''
are regular. This is rendered by the following definition.
\begin{dfn}[equitable family of partitions]
  \label{def:equitable} 
  Suppose $V$ is a set of vertices,~$\mu>0$, and
  $\delta>0$.
  We say a family of partitions $\bcP=\bcP(k-1,\bt)$ on~$V$  is 
  \emph{$(\mu,\delta)$-equitable}  if
  \begin{enumerate}[label=\alabel]
  \item\label{it:33a} $\big|V^{(k)}\setminus\cK_k(\cP^{(1)})\big|\leq \mu |V|^k$,
  \item\label{it:33b} $\cP^{(1)}=\{V_i\colon i\in[t_1]\}$ satisfies
    $|V_1|\leq\dots\leq|V_{t_1}|\leq|V_1|+1$, 
  \item\label{it:33c} for all $K\in\cK_k(\cP^{(1)})$ the complex $\bhP_K$
        (see~\eqref{eq:polyad-complex}) is a $(\delta,\bd)$-regular 
        $(k,k-1)$-complex for $\bd=(1/t_2,\dots,1/t_{k-1})$, and
  \item\label{it:33d} 
		for every $j\in [k-1]$ and for every $K\in \cK_k(\cP^{(1)})$ we have
				\[
		(1-\mu)\prod_{i=1}^j \left(\frac{1}{t_i}\right)^{\binom{k}{i}} n^k\le
		|\cK_k(\hat{P}^{(j)}_K)|\le 
		(1+\mu)\prod_{i=1}^j \left(\frac{1}{t_i}\right)^{\binom{k}{i}} n^k\,.
		\]
		  \end{enumerate}
\end{dfn}
This concludes the discussion of the auxiliary underlying structure provided by the hypergraph regularity lemma.

\subsection{Regularity lemma and embedding lemma} It is left to describe the regular properties the given $k$-uniform hypergraph $H=(V,E)$ may have with respect to the partition. Roughly speaking, $H$ will be regular
for most polyads $\hP^{(k-1)}\in\hcP^{(k-1)}$. However, for the intended application of the embedding lemma (see Theorem~\ref{thm:EL} below) we will need a refined version of the notion defined 
in~\eqref{eq:hreg1}.
\begin{dfn}[$(\delta_k,d,r)$-regular]
	\label{def:hregr}
	Let $\delta_k>0$, $d\geq 0$, and $r\in\NN$.
	We say a $k$-uniform hypergraph $H=(V,E)$ is \emph{$(\delta_k,d,r)$-regular} 
	w.r.t.~a $k$-partite $(k-1)$-uniform hypergraph~$\hP^{(k-1)}$ with $V(\hP^{(k-1)})\subseteq V$ if 
	for every collection $(Q_1,\dots,Q_r)$ of subhypergraphs $Q_s\subseteq \hP^{(k-1)}$
	we have 
	\[
		\bigg|\Big|E\cap\bigcup_{s\in[r]}\cK_{k}(Q_s)\Big|-d\,\Big|\!\bigcup_{s\in[r]}\cK_k(Q_s)\Big|\bigg|
	\leq 
	\delta_k\,\big|\cK_k(\hP^{(k-1)})\big|\,.
	\]
\end{dfn}
For $r=1$ this definition coincides with the one in~\eqref{eq:hreg1}. However, for larger~$r$
Definition~\ref{def:hregr} gives a more control over the distribution of the edges of~$H$ in $\cK_k(\hP^{(k-1)})$.
In particular, we may consider (many) subhypergraphs $Q_s$ each individually 
spanning significantly less than $|\cK_k(\hP^{(k-1)})|$ $k$-cliques and still obtain some information 
of the distribution of the edges of $H$ on such a collection $(Q_s)_{s\in[r]}$ on average.
For the proof of the hypergraph regularity lemma the parameter $\delta_k$ is required to be a fixed constant, 
but~$r$ (and the parameter~$\delta$ controlling the regularity of the underlying partition) can be given as a function of the size of the equitable partition, i.e.,~$r$ may depend on 
$(t_1,\dots,t_{k-1})$. This turned out to be useful for the proof of the embedding lemma given in~\cite{NRS06}.
Subsequently it turned out that regularity with $r=1$ is sufficient for the proof of the so-called 
counting/embedding lemma for $3$-uniform hypergraphs (see, e.g.,~\cite{NPRS09}). However, for $k>3$ 
(which we are concerned here with) such a ``simplification'' is still work in progress~\cite{NRS}.
We now are ready to state the hypergraph regularity lemma from~\cite{RSch07RL}*{Theorem~2.3}.
\begin{thm}[Regularity lemma]
	\label{thm:RL}
  For every $k\geq 2$, $\mu>0$, $\delta_k>0$, and for all functions
  $\delta\colon \NN^{k-1}\to(0,1]$ and $r\colon \NN^{k-1}\to\NN$
  there are integers~$T_0$ and~$n_0$ such that
  the following holds for every $k$-uniform hypergraph~$H=(V,E)$ on $|V|=n\geq n_0$ vertices.
  
  There is a $\bt=(t_1,\dots,t_{k-1})\in\NN^{k-1}_{>0}$ and 
  family of partitions $\bcP=\bcP(k-1,\bt)$ satisfying
  \begin{enumerate}[label=\rmlabel]
  \item $\bcP$ is $T_0$-bounded and $(\mu,\delta(\bt))$-equitable and
  \item for all but at most $\delta_k\,|\cK_k(\cP^{(1)})|$ sets $K\in\cK_k(\cP^{(1)})$ the 
  	hypergraph $H$ is $(\delta_k,d_K,r(\bt))$-regular w.r.t.\ the polyad 
	$\hP^{(k-1)}_K\in\hcP^{(k-1)}$ where $d_K=d(H\,|\,\hP^{(k-1)}_K)$.
  \end{enumerate}
\end{thm}

Notice that part~\ref{it:33d} of Definition~\ref{def:equitable}
is not part of the statement of the hypergraph
regularity lemma from~\cite{RSch07RL}. However, in applications it is often helpful and provided that the function
$\delta$ decreases sufficiently fast it is actually a consequence of properties~\ref{it:33b} and~\ref{it:33c} 
and the so-called \emph{dense counting lemma} from~\cite{KRS02} (see also~\cite{RSch07CL}*{Theorem~2.1}).

Finally, we state a consequence of the (general) counting lemma accompanying Theorem~\ref{thm:RL}, which allows 
to embed $k$-uniform hypergraphs of given isomorphism type~$F$ into~$H$. We only state a variant of this lemma   
suited for the proof of Proposition~\ref{prop:main}, i.e., 
specialised for embedding the three-edge hypergraph $F^{(k)}$ on $k+1$ vertices in sufficiently regular blocks from the partition provided by the regularity lemma. 
This result follows from~\cite{RSch07CL}*{Theorem~1.3}.
\begin{thm}[Embedding lemma]
	\label{thm:EL}
	For $k\geq 2$ and $d_k>0$ there exists $\delta_k>0$ and there are functions 
	$\delta\colon \NN^{k-2}\to(0,1]$, $r\colon \NN^{k-2}\to\NN$, and   $N\colon \NN^{k-2}\to\NN$
	such that the following holds for every $\bt=(t_2,\dots,t_{k-1})\in\NN^{k-2}_{>0}$.
	
	Suppose $\bP=(P^{(2)},\dots,P^{(k-1)})$ is a $\big(\delta(\bt),(1/t_2,\dots,1/t_{k-1})\big)$-regular
	$(k-1,k+1)$-complex with vertex partition $V_1\dcup\dots\dcup V_{k+1}$ and 
	$|V_1|,\dots,|V_{k+1}|\geq N(\bt)$ and suppose~$H$ is \mbox{a~$k$-uniform} $(k+1)$-partite 
	hypergraph on the same vertex partition such that for each of the three choices of~$a$ and $b$ with $k-1\leq a<b\leq k+1$
	there is some $d_{a,b}\geq d_k$ for which~$H$ is~$(\delta_k,d_{a,b},r(\bt))$-regular w.r.t.\ $P^{(k-1)}[V_1,\dots,V_{k-2},V_a,V_b]$.
		
	Then $H$ contains a copy of $F^{(k)}$ with vertices $v_i\in V_i$ for $i=1,\dots,k+1$
	and edges of the form $v_1\dots v_{k-2}v_av_b$ for $k-1\leq a<b\leq k+1$.  
\end{thm}
In the application of Theorem~\ref{thm:EL} the complex $\bP$ will be given by a suitable collection
of polyads from the regular partition given by Theorem~\ref{thm:RL}.
We remark that the regularity lemma also allows the functions $\delta(\cdot)$ and $r(\cdot)$ 
to depend on~$t_1$. However, this will be of no use here and is not required for the application 
of the embedding lemma.

For the proof of Proposition~\ref{prop:main} we consider a 
$(2^{1-k}+\eps,\eta,[k]^{(k-2)})$-dense hypergraph~$H$
and apply the regularity lemma to it. 
The main part of the proof concerns the appropriate selection of dense and regular polyads, that are ready for an application of the embedding lemma. This will be achieved by Proposition~\ref{prop:reduced}, which is proved 
in Sections~\ref{sec:proof}--\ref{sec:triangle}.
Proposition~\ref{prop:reduced} relies on the notion of \emph{reduced hypergraphs}
appropriate for our situation, which is the focus of the next section.

\section{Reduction to reduced hypergraphs}
\label{sec:rreduced}

As in~\cites{RRS-a,RRS-b,RRS-c,RRS-d} we will use 
the hypergraph regularity method for transforming the problem at hand into a somewhat 
different problem that speaks about certain ``reduced hypergraphs,'' 
that are going to be introduced next (see Definition~\ref{dfn:reduced} below). 
The assumption of $[k]^{(k-2)}$-denseness in Proposition~\ref{prop:main} 
allows us to work with the following concept.

\begin{dfn}
\label{dfn:reduced}
Suppose that we have a finite index set $I$ and for each $x\in I^{(k-1)}$ a finite 
nonempty vertex set $\cP_x$ such that for any two distinct $x, x'\in I^{(k-1)}$
the sets $\cP_x$ and $\cP_{x'}$ are disjoint. 
Assume further that for any $y\in I^{(k)}$ we have a $k$-uniform $k$-partite hypergraph~$\cA_y$ with vertex partition $\bigdcup_{x\in y^{(k-1)}}\cP_x$.
Then the $k$-uniform $\binom{|I|}{k-1}$-partite hypergraph~$\cA$ with
\[
	V(\cA)=\bigdcup_{x\in I^{(k-1)}}\cP_x
	\qand
	E(\cA)=\bigcup_{y\in I^{(k)}}\cA_y
\]
is a {\it reduced $k$-uniform hypergraph}. We also refer to $I$ as the {\it index set} 
of~$\cA$, to the sets~$\cP_x$ as the {\it vertex classes} of~$\cA$, and to the $\binom{|I|}{k}$
hypergraphs~$\cA_y$ as the {\it constituents} of~$\cA$.
\end{dfn}

In our context the reduced hypergraph $\cA$ encodes (a suitable collection of)
dense and regular polyads of a family of partitions provided by the 
regularity lemma applied to a hypergraph $H$. In fact, the vertex classes $\cP_x$ 
shall correspond to the $t_{k-1}$ different $(k-1)$-uniform $(k-1)$-partite 
hypergraphs that ``belong'' to a given polyad $\hP^{(k-2)}\in \hcP^{(k-2)}$ 
(see~\eqref{eq:equit2}). 
Moreover, a collection of~$k$ vertices, each from a different vertex class of a 
constituent of $\cA$, will then correspond to a $(k-1)$-uniform $k$-partite 
polyad $\hP^{(k-1)}$ from the family of partitions, and 
an edge of the constituent will signify that~$H$ is sufficiently dense and regular 
on this polyad.
As it will turn out below, the assumption 
that the hypergraph $H$ in Proposition~\ref{prop:main} is
$\bigl(2^{1-k}+\eps, \eta, [k]^{(k-2)}\bigr)$-dense can be ``translated'' into a 
density condition applying to the constituents of the reduced hypergraph that we 
obtain via regularisation.

\begin{dfn}
\label{dfn:red-dense}
Given a real number $d\in [0,1]$ and a reduced $k$-uniform hypergraph~$\cA$ 
with index set $I$, we say that $\cA$ is 
{\it $d$-dense} 
provided that 
\[
	e(\cA_y)\ge d\cdot \prod_{x\in y^{(k-1)}}|\cP_x|
\]
holds for all $y\in I^{(k)}$.
\end{dfn}

Next we need to tell which configuration that might appear in a reduced hypergraph 
corresponds (in view of the embedding lemma) to an $F^{(k)}$ in the original hypergraph.

\begin{dfn}
Let $\cA$ be a reduced $k$-uniform hypergraph with index set $I$. A set $z\in I^{(k+1)}$
{\it supports an $F^{(k)}$} if for every $x\in z^{(k-1)}$ 
one can select a $P_x\in \cP_x$ such that there are 
at least 
three sets $y\in z^{(k)}$ satisfying
\begin{equation} \label{eq:Fk}
	\bigl\{P_x\colon x\in y^{(k-1)}\bigr\}\in E\bigl(\cA_y\bigr)\,.
\end{equation}
\end{dfn}

The following alternative description of $(k+1)$-sets supporting an $F^{(k)}$
will turn out to be useful in Section~\ref{sec:proof}. 

\begin{fact}\label{fact:revision}
Suppose that $\cA$ is a reduced $k$-uniform hypergraph with index set $I$. 
A set~${z\in I^{(k+1)}}$ supports an $F^{(k)}$ if there exist distinct $k$-sets
$y_1, y_2, y_3\in z^{(k)}$ and edges $e_1\in E\bigl(\cA_{y_1}\bigr)$, 
$e_2\in E\bigl(\cA_{y_2}\bigr)$, and $e_3\in E\bigl(\cA_{y_3}\bigr)$
no two of which are disjoint. 
\end{fact}

\begin{proof}
We intend to choose vertices $P_x\in\cP_x$ for $x\in z^{(k-1)}$ such that 
\begin{equation} \label{eq:Fk2}
	e_i= \bigl\{P_x\colon x\in y_i^{(k-1)}\bigr\}
\end{equation}
holds for $i=1,2,3$. Notice that if $x$ does not belong to
$y_1^{(k-1)}\cup y_2^{(k-1)}\cup y_3^{(k-1)}$ the choice of~$P_x$ is
immaterial. Moreover, if $x$ belongs to exactly one of the sets 
$y_1^{(k-1)}$, $y_2^{(k-1)}$, and~$y_3^{(k-1)}$, then the corresponding 
instance of~\eqref{eq:Fk2} determines $P_x$ uniquely. 

It remains to check that if $x$ belongs to at least two of these sets, 
then the demands imposed on $P_x$ by~\eqref{eq:Fk2} do not contradict each other.

Now suppose, for instance, that $x\in y_1^{(k-1)}\cap y_2^{(k-1)}=(y_1\cap y_2)^{(k-1)}$.
Owing to ${|y_1\cap y_2|=k-1}$ this implies $x=y_1\cap y_2$. Let $a_3$ denote an arbitrary
vertex from $e_1\cap e_2$ and let $\cP_{\overline{x}}$ be the vertex class of $\cA$ 
containing $a_3$. Because of $e_1\in E\bigl(\cA_{y_1}\bigr)$ and $e_2\in E\bigl(\cA_{y_2}\bigr)$
we have $\overline{x}\in y_1^{(k-1)}\cap y_2^{(k-1)}$ and, consequently, 
$\overline{x}=y_1\cap y_2=x$. This shows that it is legitimate to set $P_x=a_3$ and the 
proof of Fact~\ref{eq:Fk2} is complete (see also Figure~\ref{fig:1}).   
\end{proof}

\begin{figure}[ht]
\begin{center}
\includegraphics{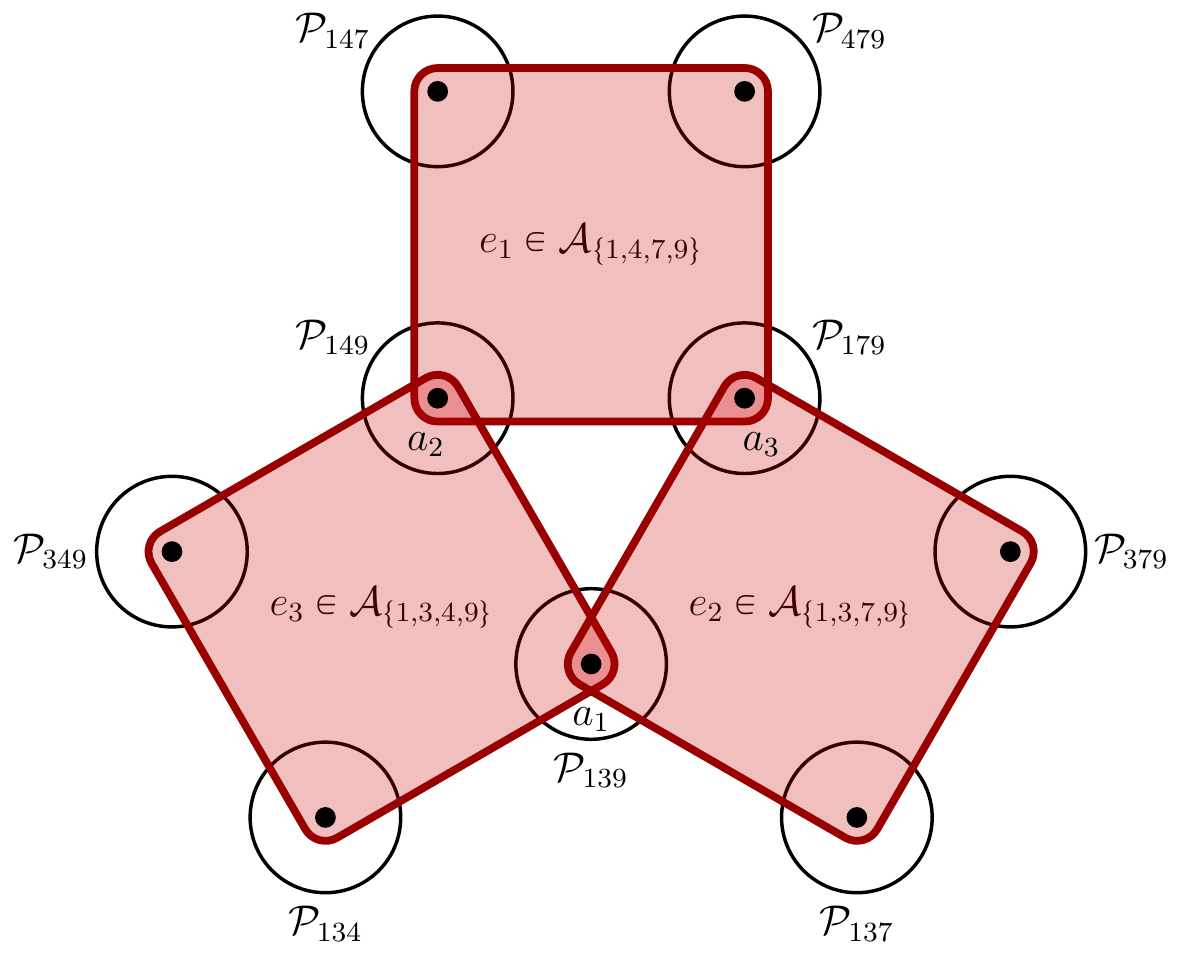}
\end{center}
\caption{$k=4$, $z=\{1,3,4,7,9\}$, $y_1=z\setminus\{3\}$, $y_2=z\setminus\{4\}$, 
and $y_3=z\setminus\{7\}$.}
\label{fig:1}
\end{figure}
We are now ready to formulate a statement about reduced hypergraphs to which
Proposition~\ref{prop:main} reduces  in the light of the hypergraph regularity method.

\begin{prop} 
\label{prop:reduced}
For every $\eps>0$ there exists a positive integer $m$ such that every 
${\bigl(2^{1-k}+\eps\bigr)}$-dense, reduced $k$-uniform hypergraph
with index set of size at least~$m$ supports an $F^{(k)}$.
\end{prop}

In the rest of this section we shall show that this statement does indeed imply
our main result. The three subsequent sections will then deal with the proof of 
Proposition~\ref{prop:reduced}.

\begin{proof}[Proof of Proposition~\ref{prop:main} assuming Proposition~\ref{prop:reduced}]
Given $\eps>0$ we have to define $\eta>0$ and $n_0\in\NN$ with the desired property. 
We divide the argument that follows into four steps.

\subsection*{Step 1: Selection of constants.} 
We commence by picking some auxiliary constants  
\begin{equation}\label{eq:dk-mu}
	d_k\ll \eps \qand \mu\ll m^{-1}\ll \eps\,.
\end{equation}
With $d_k$ we appeal to the embedding lemma, i.e., Theorem~\ref{thm:EL}, and it yields
a constant $\delta_k$ and functions 
$\delta\colon \NN^{k-2}\to(0,1]$, $r\colon \NN^{k-2}\to\NN$, and  $N\colon \NN^{k-2}\to\NN$.
We need some further constants
\begin{equation} \label{eq:delta'}
	\delta'_k\ll \xi \ll \delta_k, m^{-1}
\end{equation}
that depend solely on $\delta_k$ and $m$.

Next we deliver $\mu$, $\delta'_k$, and the functions
\[
	\widetilde{\delta}\colon \NN^{k-1}\longrightarrow (0, 1]\,, \quad
	(t_1, \ldots, t_{k-1})\longmapsto \delta(t_2, \ldots, t_{k-1})
\]
as well as 
\[
	\widetilde{r}\colon \NN^{k-1}\longrightarrow \NN\,,\quad
	(t_1, \ldots, t_{k-1})\longmapsto r(t_2, \ldots, t_{k-1})
\]
to the hypergraph regularity lemma, thus receiving two large integers $T_0$ and $n'_0$.
Finally we take 
\begin{equation}\label{eq:n0}
	\eta\ll T_0^{-1} \qand 
	n_0=\max\bigl(\{2T_0\cdot N(t_2, \ldots, t_{k-2})\colon t_2, \ldots, t_k\le T_0\}
					\cup\{n_0'\}\bigr)\,.
\end{equation}
Now let $H=(V, E)$ be any $\bigl(2^{1-k}+\eps, \eta, [k]^{(k-2)}\bigr)$-dense
$k$-uniform hypergraph
with ${|V|=n\ge n_0}$. We are to prove that $H$ contains a copy of $F^{(k)}$.

\subsection*{Step 2: \texorpdfstring{Selection from $\cP^{(k-2)}$}{Selecting from the partition}.} The regularity 
lemma yields a $T_0$-bounded and $(\mu,\widetilde{\delta}(\bt))$-equitable 
partition $\bcP$ of $V^{(k-1)}$ for some $\bt=(t_1,\dots,t_{k-1})\in\NN^{k-1}_{>0}$ 
such that 
\begin{enumerate}  
	\item[$(*)$] for all but at most $\delta'_k\,|\cK_k(\cP^{(1)})|$ sets $K\in\cK_k(\cP^{(1)})$ 
	the hypergraph $H$ is $(\delta'_k,d_K,\widetilde{r}(\bt))$-regular w.r.t.\ the polyad 
	$\hP^{(k-1)}_K\in\hcP^{(k-1)}$, where $d_K=d(H\,|\,\hP^{(k-1)}_K)$.
\end{enumerate}
For the rest of this proof we will simply say that $H$ is ``regular'' w.r.t.\ 
to a polyad $\hP^{(k-1)}_K\in\hcP^{(k-1)}$, when we mean
that it is $(\delta'_k,d_K,\widetilde{r}(\bt))$-regular w.r.t. it. 

The remaining part of this step is only needed when $k\ge 4$.   
For every $(k-2)$-subset~$\cW$ of~$\cP^{(1)}$ 
the set 
\[
	\cK_{k-2}(\cW)=\bigl\{J\in V^{(k-2)} \colon 
	J\cap V_{i}\ne \emptyset \text{ for every } V_i\in \cW\bigr\}
\]
is split by $\cP^{(k-2)}$ into the same number 
$t^*=\prod_{2\le \ell\le k-2}t_\ell^{\binom{k-2}{\ell}}$
of $(k-2)$-uniform hypergraphs.
Let us now pick for each such $\cW$ one of these $t^*$ hypergraphs 
as follows. For every transversal of $\cP^{(1)}$, i.e., a $t_1$-element 
set $T\subseteq V$ with $|T\cap V_i|=1$ for every $i\in[t_1]$, we consider the 
selection
\[
	\mathfrak{S}_T
	=\bigl\{P^{(k-2)}_J\in\cP^{(k-2)}\colon J\in T^{(k-2)}\bigr\}
\]	
and let
\[
	\cK_k(\mathfrak{S}_T)=\bigl\{K\in \cK_k(\cP^{(1)})\colon
	P^{(k-2)}_J\in\mathfrak{S}_T \text{ for every } J\in K^{(k-2)} \bigr\}\,.  
\]
be the collection of $k$-subsets of $V$ that are supported by~$\mathfrak{S}_T$. 

Since by Definition~\ref{def:equitable}~\ref{it:33d} all $(k-2)$-uniform 
$k$-partite polyads have the same volume up to a multiplicative factor controlled 
by $\mu$, a simple averaging argument shows that for some appropriate transversal~$T$ 
all but at most $2 \delta'_k\, |\cK_k(\mathfrak{S}_T)|$ members of $\cK_k(\mathfrak{S}_T)$ have the property that~$H$ is regular with respect to their polyad. 
From now on we fix one such choice of $T$ and the corresponding collection $\mathfrak{S}_T$.

\subsection*{Step 3: \texorpdfstring{Passing to an $[m]$-subset of $\cP^{(1)}$}{Passing to a subset of vertex classes}.} 
Notice that Definition~~\ref{def:equitable}~\ref{it:33a} and $\mu\ll m^{-1}$ yield $t_1\ge m$.

Now consider the auxiliary $k$-uniform hypergraph~$\cB$ with vertex set~$\cP^{(1)}$
having all those $k$-subsets $\cY$ of~$\cP^{(1)}$ as edges for which more than 
$\xi\, |\cK_k(\mathfrak{S}_T)\cap \cK_k(\cY)|$ members of
\[
	\cK_k(\cY)=\bigl\{K\in V^{(k)} \colon 
	K\cap V_{i}\ne \emptyset \text{ for every } V_i\in \cY\bigr\}
\]
have the property that $H$ fails to be regular w.r.t.\ their polyad, i.e., $\cY\in E(\cB)$ if
\begin{multline*}
	\big|\big\{K\in\cK_k(\mathfrak{S}_T)\cap\cK_k(\cY)\colon H\ 
	\text{ is not $(\delta'_k,d_K,\widetilde{r}(\bt))$-regular w.r.t.\ } \hP^{(k-1)}_K\big\}\big|\\
	> \xi\, |\cK_k(\mathfrak{S}_T)\cap \cK_k(\cY)|\,.
\end{multline*}
By our choice of $\mathfrak{S}_T$ and~\eqref{eq:delta'} we can achieve that~$\cB$
has at most $\xi\binom{t_1}{k}$ edges. 
Consequently an $m$-subset of~$\cP^{(1)}$ spans on average no more than $\xi\binom{m}{k}$
edges in $\cB$. In particular, an appropriate choice of $\xi\ll m^{-1}$ guarantees 
that $\cB$ has an independent set $\cM$ of size~$m$. 

We shall now define a reduced $k$-uniform 
hypergraph~$\cA$ with index set $\cM$. For every $(k-1)$-subset $\cX$ of $\cM$
the vertex class $\cP_{\cX}$ is defined to be the set of all $P^{(k-1)}\in \cP^{(k-1)}$ 
with $P^{(k-1)}\subseteq \cK_{k-1}(\cX)$ whose polyads are composed of members 
of $\mathfrak{S}_T$, i.e., $P^{(k-1)}\in\cP_{\cX}$ if for some (and hence for every) $J\in E(P^{(k-1)})$
we have
\[
	P^{(k-2)}_I\in\mathfrak{S}_T\ \text{for every $I\in J^{(k-2)}$.}
\]
As a consequence all the vertex classes $\cP_{\cX}$ have the same size $t_{k-1}$. 
It remains to define the constituents of $\cA$. Given a $k$-subset~$\cY$ of $\cM$
we let $E(\cA_\cY)$ be the collection of all $k$-subsets of $\bigcup_{\cX\in \cY^{(k-1)}}\cP_\cX$
that form a $(k-1)$-uniform $k$-partite polyad w.r.t.\ which $H$ is regular and has at least the density $d_k$.

As we will show in our last step, the reduced hypergraph 
\begin{equation}\label{eq:Adense}
	\cA \text{ is } (2^{1-k}+\eps/2)\text{-dense.}
\end{equation}
Owing to $m^{-1}\ll \eps$ and Proposition~\ref{prop:reduced} this will imply 
that~$\cA$ supports an~$F^{(k)}$ and by the definition of $\cA$ this configuration
corresponds to a $\big(\widetilde{\delta}(\bt),(1/t_2,\dots,1/t_{k-1})\big)$-regular
$(k-1, k+1)$-complex on which $H$ is sufficiently dense and regular for the embedding lemma
to be applicable. Moreover, \eqref{eq:n0} and Definition~\ref{def:equitable}~\ref{it:33b}
imply
\[
	|V_i|\ge \frac{n}{2t_1}\ge\frac{n}{2T_0}\ge N(t_2, \ldots, N_{t_{k-2}})
\]
for all $i\in [t_1]$, meaning that the vertex classes of this complex are also 
sufficiently large. Altogether this shows that $H$ contains indeed an $F^{(k)}$
provided that~\eqref{eq:Adense} is true.

\subsection*{Step 4: \texorpdfstring{Verifying~\eqref{eq:Adense}}{Verifying denseness}.}
Given any $k$-subset $\cY$ of $\cM$ we are to prove that 
\[
e(\cA_\cY)\ge (2^{1-k}+\eps/2)t_{k-1}^k\,.
\]
Now, since $H$ is $\bigl(2^{1-k}+\eps, \eta, [k]^{(k-2)}\bigr)$-dense, we know that
\begin{equation} \label{eq:H1}
	(2^{1-k}+\eps)|\cK_k(\cY)\cap \cK_k(\mathfrak{S}_T)|-\eta n^k \le
	|\cK_k(\cY)\cap \cK_k(\mathfrak{S}_T)\cap e(H)|\,.
\end{equation}
By Definition~\ref{def:equitable}~\ref{it:33d} 
every polyad $\hP^{(k-1)}$ satisfies
\begin{equation}\label{eq:underk}
	|\cK_k(\hP^{(k-1)})|
	=
	(1\pm\mu)\prod_{i=1}^{k-1}\left(\frac{1}{t_i}\right)^{\binom{k}{i}}n^k\,.
\end{equation}
and for the $(k-2)$-uniform $k$-partite polyad defined by the selection
$\mathfrak{S}_T$ restricted to the vertex classes in $\cY$
we have 
\begin{equation}\label{eq:under}
	|\cK_k(\cY)\cap \cK_k(\mathfrak{S}_T)|
	=
	(1\pm\mu)\prod_{i=1}^{k-2}\left(\frac{1}{t_i}\right)^{\binom{k}{i}}n^k\,.
\end{equation}
Combining the lower bound in~\eqref{eq:under} with our choice
$\eta\ll T_0^{-1}, \eps$
leads to
\[
	|\cK_k(\cY)\cap \cK(\mathfrak{S}_T)|
	\ge \frac{n^k}{T_0^{2^k}}\ge \frac{6\eta n^k}{\eps}
\]
and hence~\eqref{eq:H1} rewrites as
\begin{equation} \label{eq:H2}
	\bigl(2^{1-k}+\tfrac{5\eps}6\bigr)|\cK_k(\cY)\cap \cK_k(\mathfrak{S}_T)| \le
	|\cK_k(\cY)\cap \cK_k(\mathfrak{S}_T)\cap e(H)|\,.
\end{equation}
Among the edges of $H$ counted on the right-hand side there may be some belonging
to polyads w.r.t.\ which $H$ fails to be regular, but by our choice of $\cM$ in the 
third step and by $\cY\subseteq \cM$ their number can be at most 
$\xi\, |\cK_k(\mathfrak{S}_T)\cap \cK_k(\cY)|$. Moreover at most 
$d_k |\cK_k(\mathfrak{S}_T)\cap \cK_k(\cY)|$ edges from 
$\cK_k(\cY)\cap \cK_k(\mathfrak{S}_T)\cap e(H)$ can be supported by polyads 
with respect to which $H$ has at most the density $d_k$. The other
edges from this set are supported by polyads that are encoded as edges of $\cA_\cY$.
Conversely
any polyad~$\hP^{(k-1)}$ can support at most
\[
|\cK_{k}(\hP^{(k-1)})|
\overset{\eqref{eq:underk}}{\leq} 
(1+\mu)\frac{1}{t_{k-1}^k}\prod_{i=1}^{k-2}\left(\frac{1}{t_i}\right)^{\binom{k}{i}}n^k
\overset{\eqref{eq:under}}{\leq}
\frac{1+\mu}{1-\mu}\cdot \frac{|\cK_k(\cY)\cap \cK_k(\mathfrak{S}_T)|}{t_{k-1}^k}
\]
edges of $H$. For these reasons~\eqref{eq:H2}
leads to
\[
	\bigl(2^{1-k}+\tfrac{5\eps}6-\xi-d_k\bigr)|\cK_k(\cY)\cap \cK_k(\mathfrak{S}_T)| \le
	e(\cA_{\cY})\cdot \frac{1+\mu}{1-\mu}\cdot \frac{|\cK_k(\cY)\cap \cK_k(\mathfrak{S}_T)|}{t_{k-1}^k}\,.
\]
Using $\xi, d_k\le \tfrac\eps{12}$ this yields
\[
	\frac{1-\mu}{1+\mu}\cdot \bigl(2^{1-k}+\tfrac{2\eps}3\bigr)t_{k-1}^k
	\le e(\cA_\cY)\,.
\]
So an appropriate choice of $\mu$ at the beginning of the proof leads indeed 
to the desired result.
\end{proof}

\section{Towards the proof of Proposition~\ref{prop:reduced}} 
\label{sec:proof}

Up to two purely graph theoretic results deferred to later sections, 
we will give the proof of Proposition~\ref{prop:reduced} in this section.
Let us begin with a brief description of two of the ideas appearing in this proof.

\begin{enumerate}
\item[$\bullet$]
The first observation is that rather than studying the constituents of the 
reduced hypergraph $\cA$ under consideration directly,
it suffices to deal with certain bipartite graphs obtained by projection.
Essentially, finding an $F^{(k)}$ in $\cA$ amounts to the same thing as finding 
a triangle in a multipartite graph that is composed in an appropriate way 
of such bipartite projections. This step of the argument will be rendered by a ``triangle
lemma'' (see Theorem~\ref{thm:triangle} below), which roughly tells us that if a 
large number of sufficiently ``rich'' bipartite graphs interact, then they necessarily 
create a triangle.

\item[$\bullet$]  
Now irrespective of what such a triangle lemma says precisely, there arises the 
question why many of these bipartite projections will in fact be ``rich''. 
Ultimately, of course, this must be a consequence of our density assumption imposed 
on $\cA$. More precisely, we will prove a so-called ``path lemma'' 
(see Theorem~\ref{thm:path} below) stating that long concatenations of ``poor'' 
bipartite graphs will always contain fewer paths than what we would expect in view
of the density of~$\cA$. From this it will follow, e.g., 
that every constituent of~$\cA$ admits at least one ``rich'' projection. 
Once they are found, these ``rich'' projections will be assembled in a manner that is ready 
for an application of the triangle lemma by means of some Ramsey theoretic arguments.    
\end{enumerate}
   
In some sense it does not matter for the proof described in this section what the terms
``rich'' and ``poor'' used informally in the above discussion actually mean: only the 
path lemma and the triangle lemma are real. But to aid the readers orientation it might
still be helpful to say now for which such concepts we will later show that those two
statements are true. 
   
\begin{dfn} 
Let $\xi>0$ and let $G$ be a bipartite graph with fixed ordered bipartition~$(X, Y)$.
We say that $G$ is {\it $\xi$-poor} if there are at most $\xi\,|Y|$ many vertices 
$y\in Y$ for which the number of two-edge walks in $G$ starting at $y$ is larger than 
$\bl\tfrac 14+\xi\br|X|\,|Y|$. Otherwise $G$ is said to be {\it $\xi$-rich}.
\end{dfn}

Note, that these definitions concern ordered bipartitions $(X,Y)$ and hence they are not symmetric.
Moreover, the walks we consider may use one edge twice. 
This means that if $xy, xy'\in E(G)$ holds for some three vertices $x$, $y$, and $y'$
of $G$, then $yxy'$ is regarded as an two-edge walk starting at $y$ irrespective of
whether $y\ne y'$ holds or not.
  
The following result will be proved in Section~\ref{sec:path}. It will be used below 
for locating many rich graphs among the projections of the constituents of a 
${\bigl(2^{1-k}+\eps\bigr)}$-dense reduced hypergraph.

\begin{thm}[Path lemma]
\label{thm:path} 
Given $\eps>0$ and a positive integer $k$, there exists a positive real 
number~$\xi$ for which the following holds: 
If $G$ is a $k$-partite graph with nonempty vertex
classes $V_1, \ldots, V_{k}$ such that for all $r\in [k-1]$ the graph $G[V_r, V_{r+1}]$
is~$\xi$-poor, then there are less than
\[
	\bl\tfrac 1{2^{k-1}}+\eps\br\prod_{i=1}^{k}|V_i|
\]
many $k$-tuples $(v_1, \ldots, v_k)\in V_1\times\ldots\times V_{k}$ 
for which $v_1v_2\ldots v_{k}$ is a path in $G$.
\end{thm}

Next we state the triangle lemma, whose proof is deferred to Section~\ref{sec:triangle}.
 
\begin{thm}[Triangle lemma] 
\label{thm:triangle}
If $m^{-1}\ll\xi$, then every $m$-partite graph $G$ with nonempty vertex classes 
$V_1, \ldots, V_m$ such that for all $i$ and $j$ with $1\le i<j\le m$ the bipartite graphs~$G[V_i, V_j]$ 
are $\xi$-rich contains a triangle.
\end{thm}

Now everything is in place for the main goal of the present section.

\begin{proof}[Proof of Proposition~\ref{prop:reduced} assuming Theorems~\ref{thm:path}
and \ref{thm:triangle}]
Let us start with the hierarchy
\[
	m^{-1}\ll m_*^{-1}\ll \xi\ll \eps\,.
\]
It suffices to show that any 
${\bigl(2^{1-k}+\eps\bigr)}$-dense, 
reduced $k$-uniform hypergraph $\cA$ with index set $[m]$ contains an $F^{(k)}$.
As usual we let 
\[
	\bigl\{\cP_x\colon x\in [m]^{(k-1)}\bigr\}
	\qand 
	\bigl\{\cA_y\colon y\in [m]^{(k)}\bigr\}
\]
denote the collections of vertex classes and constituents of $\cA$, respectively.

Consider an arbitrary $y\in [m]^{(k)}$ and let $y=\{i_1, \ldots, i_k\}$ list its
elements in increasing order. 
We associate with $y$ a certain $k$-partite graph $G^y$ with vertex classes 
$V^y_1, \ldots, V^y_k$, where $V^y_r=\cP_{{y-\{i_r\}}}$ for all $r\in [k]$.
The edges of $G^y$ between two consecutive vertex classes $V^y_r$ and $V^y_{r+1}$ 
with $1\le r<k$ are defined by projection as follows: 
for $a\in V^y_r$ and $b\in V^y_{r+1}$ we draw an edge between $a$ and $b$ in $G^y$ 
if and only if there is an edge of $\cA_y$ containing both $a$ and $b$. 
Now there is an obvious injective map from the edges of $\cA_y$ to the paths 
$(v_1, \ldots, v_k)\in V^y_1\times\ldots\times V^y_k$ in $G^y$, 
and hence there are at least $\bl 2^{1-k}+\eps\br\prod_{r=1}^k |V^y_r|$ such paths. 
Thus the path lemma (Theorem~\ref{thm:path}) tells us that for at least one value of $r\in[k-1]$ the bipartite 
graph $G^y[V^y_r, V^y_{r+1}]$ is $\xi$-rich. Let us denote one such possible value of 
$r$ by $h(y)$.

As the construction described in the foregoing paragraph applies to every $k$-subset~$y$
of~$[m]$, we have thereby defined a function 
\[
	h\colon[m]^{(k)}\lra [k-1]\,.
\]
Due to Ramsey's theorem and $m\gg m_*$, there exists an $m_*$-subset $Q$ of $[m]$ 
together with some $r\in [k-1]$ such that $h(y)=r$ holds for all $y\in Q^{(k)}$.
We will show in the sequel that some $z\in Q^{(k+1)}$ supports an $F^{(k)}$,
so for notational transparency we may suppose $Q=[m_*]$ from now on.   

At this moment we may already promise that the set
\[
	z^-=\{1, 2, \ldots, r-1\}\cup \{m_*+r-k+2, \ldots, m_*\}
\]
will be a subset of the desired set $z$. Since $|z^-|=k-2$,
this means that we will need to find three further indices $t_1$, $t_2$, and $t_3$ 
from the interval $J=[r, m_*+r-k+1]$ such that the set $z=z^-\cup\{t_1, t_2, t_3\}$
supports an $F^{(k)}$.

To this end we construct an auxiliary $|J|$-partite graph $G$. Its collection of 
vertex classes is going to be $\bigl\{\cP_{z^-\cup\{t\}} \colon t\in J\bigr\}$
and it remains to specify the set of edges of $G$. Notice that for any~$t_1<t_2$ from $J$ 
the $r$-th and $(r+1)$-st member of the set $z^-\cup\{t_1, t_2\}$ in its increasing 
enumeration are $t_1$ and $t_2$ respectively, whence  
$V^{z^-\cup\{t_1, t_2\}}_{r}=\cP_{z^-\cup\{t_2\}}$ and
$V^{z^-\cup\{t_1, t_2\}}_{r+1}=\cP_{z^-\cup\{t_1\}}$. 
We may thus complete the definition of $G$ by stipulating 
\[
	G\bigl[\cP_{z^-\cup\{t_1\}}, \cP_{z^-\cup\{t_2\}}\bigr]
	=G^{z^-\cup\{t_1, t_2\}}\bigl[V^{z^-\cup\{t_1, t_2\}}_{r+1}, 
									V^{z^-\cup\{t_1, t_2\}}_{r}\bigr]
\]
whenever $t_1<t_2$ are from $J$. Owing to our choice of $r$, the multipartite 
graph $G$ has the property that all its bipartite parts 
$G\bigl[\cP_{z^-\cup\{t_2\}}, \cP_{z^-\cup\{t_1\}}\bigr]$ 
with $t_1<t_2$ are \mbox{$\xi$-rich}. As we still have $|J|=m_*-k+2\gg \xi^{-1}$,
the triangle lemma is applicable to $G$. Therefore, Theorem~\ref{thm:triangle} tells us that some three vertices of $G$,
say $a_1\in \cP_{z^-\cup\{t_1\}}$, $a_2\in \cP_{z^-\cup\{t_2\}}$, and 
$a_3\in \cP_{z^-\cup\{t_3\}}$, form a triangle. Of course,
$t_1, t_2, t_3\in J$ are distinct.

Utilising Fact~\ref{fact:revision} we are now going to verify that the 
set $z=z^{-}\cup\{t_1, t_2, t_3\}$ supports an $F^{(k)}$. To this end we
set $y_i=z\setminus \{t_i\}$ for $i=1, 2, 3$. Moreover, we recall 
that the edge $a_2a_3$ of $G$ indicates that there is an edge
$e_{1}\in E\bigl(\cA_{y_1}\bigr)$ containing $a_2$ and $a_3$. 
Similarly the edges $a_1a_3$ and $a_1a_2$ of $G$ lead to certain edges 
$e_{2}$ and~$e_{3}$ with $a_1, a_3\in e_{2}\in E\bigl(\cA_{y_2}\bigr)$ and
$a_1, a_2\in e_{3}\in E\bigl(\cA_{y_3}\bigr)$, respectively. 
Due to $a_1\in e_2\cap e_3$, $a_2\in e_1\cap e_3$, and $a_3\in e_1\cap e_2$
these edges have the required properties. 
\end{proof}

\section{The path lemma} 
\label{sec:path}

In this section we are concerned with proving the path lemma. We will actually obtain 
a slightly stronger statement (see Proposition~\ref{prop:strong-path} below) that seems
to be easier to show by induction on $k$. The lemma that follows encapsulates what happens 
in the inductive step.

\begin{lemma}
\label{lem:51}
Let $\xi>0$ and $M\ge 0$ denote two real numbers. 
Suppose that 
\begin{enumerate}
\item[$\bullet$] $G$ is a $\xi$-poor bipartite graph with bipartition $(X, Y)$,
\item[$\bullet$] and that $f\colon Y\lra [0, M]$ is a function.
\end{enumerate}
Then 
\begin{enumerate}
\item[$\bullet$] the real number $a\ge 0$ with $\sum_{y\in Y}f(y)^2=|Y|\,a^2$
\item[$\bullet$] and the function $g\colon X\lra \RR$ defined by 
					$g(x)=\sum_{y\in N(x)}f(y)$ for all $x\in X$ 
\end{enumerate}
satisfy
\[
	\sum_{x\in X}g(x)^2\le \left(\bl \tfrac 14+\xi\br a^2+\xi\,M^2\right)|X|\,|Y|^2\,.
\] 
\end{lemma}

\begin{proof} 
For every $x\in X$ the Cauchy-Schwarz inequality yields
\[
	g(x)^2
	=
	\bigg(\sum_{y\in N(x)}f(y)\bigg)^2
	\leq 
	d(x)\sum_{y\in N(x)}f(y)^2\,.
\]
Summing over all $x\in X$ leads to
\[
	\sum_{x\in X}g(x)^2
	\le
	\sum_{x\in X}\bigg(d(x)\sum_{y\in N(x)}f(y)^2\bigg)
	=
	\sum_{xy\in E(G)} d(x)f(y)^2=\sum_{y\in Y}\Bigl(\sum_{x\in N(y)}d(x)\Bigr)f(y)^2\,.
\]
Now for every $y\in Y$ the expression 
\[
	P_y=\sum_{x\in N(y)}d(x)
\]
counts the number of two-edge walks of $G$ starting at $y$, including degenerate ones.
With this notation the above inequality rewrites as
\[
	\sum_{x\in X}g(x)^2\le \sum_{y\in Y}P_yf(y)^2\,.
\]
The $\xi$-poorness of~$G$ tells us that the set
\[
	A=\left\{y\in Y \colon P_y> \bl\tfrac 14+\xi\br|X|\,|Y|\right\}
\]
has at most the size $\xi\,|Y|$. It is also clear that $P_y\le |X|\,|Y|$ holds for all $y\in Y$. 
Hence
\begin{align*}
	\sum_{x\in X}g(x)^2 &\le \sum_{y\in Y-A}P_yf(y)^2+ \sum_{y\in A}P_yf(y)^2 \\
	&\le \bl\tfrac 14+\xi\br|X|\,|Y|\sum_{y\in Y}f(y)^2+|A|\,|X|\,|Y|\,M^2 \\
	&\le \left(\bl \tfrac 14+\xi\br a^2+\xi\,M^2\right)|X|\,|Y|^2\,,
\end{align*}
which is what we wanted to show.
\end{proof}

\begin{prop}
\label{prop:strong-path}
Given $\eps>0$ and a positive integer $k$, 
there exists some $\xi>0$ with the following property: 
let $G$ be a $k$-partite graph with nonempty vertex classes $V_1, \ldots, V_{k}$ 
such that $G[V_r, V_{r+1}]$ is $\xi$-poor for all $r\in [k-1]$. 
Denote for each $x\in V_1$ the number of $(k-1)$-tuples 
$(v_2, \ldots, v_{k})\in V_2\times\ldots\times V_{k}$ such that $xv_2\ldots v_{k+1}$ 
is a path in $G$ by $g(x)$. Then
\[
	\sum_{x\in V_1}g(x)^2< \bl\tfrac 1{2^{k-1}}+\eps\br^2|V_1|\prod_{i=2}^{k}|V_i|^2 
\]
holds.
\end{prop}  

\begin{proof}
For fixed $\eps$ we argue by induction on $k$. In the base case $k=1$ the graph $G$ 
just consists of the independent set $V_1$, the function $g$ is constant attaining 
always the value~$1$, and thus our assertion is trivially valid for any $\xi>0$.

Now let $k\ge 2$ and suppose that the proposition is already known for $k-1$ in place of~$k$, 
say with $\xi'$ in place of $\xi$. Depending on $k$, $\eps$, and $\xi'$ we let $\xi>0$ be so 
small that 
\[
	\xi\le \xi'
	\qand 
	(1+4\xi)\bl\tfrac 1{2^{k-1}}+\tfrac\eps 2\br^2+\xi<\bl\tfrac 1{2^{k-1}}+\eps\br^2
\]
hold.

To see that $\xi$ is as desired, let the $k$-partite graph $G$ with vertex classes 
$V_1, \ldots, V_{k}$ and the function $g$ be as described above. 
For each $y\in V_2$ we write $f(y)$ for the number of $(k-2)$-tuples 
$(v_3, \ldots, v_{k})\in V_3\times\ldots\times V_{k}$ such that $yv_3\ldots v_{k}$ 
is a path in $G$. 
Clearly we have 
\[
	g(x)=\sum_{y\in N(x)\cap V_2}f(y)
\]
for each $x\in V_1$. Moreover, the number 
\[
	M=\prod_{i=3}^{k}|V_i|
\]
satisfies $f(y)\le M$ for all $y\in V_2$. 
We may thus apply Lemma~\ref{lem:51} to the bipartite graph called $G[V_1, V_2]$ here
in place of $G$ there. This tells us that for the real number $a\ge 0$ defined by 
\begin{equation} \label{eq:intro-a}
	\sum_{y\in V_2}f(y)^2=|V_2|\,a^2
\end{equation}
we have 
\begin{equation} \label{eq:more-a}
	\sum_{x\in V_1}g(x)^2\le \left(\bl \tfrac 14+\xi\br a^2+\xi\,M^2\right)|V_1|\,|V_2|^2\,.
\end{equation}
Owing to $\xi\le \xi'$ the induction hypothesis yields 
\[
	\sum_{y\in V_2}f(y)^2<\bl\tfrac 1{2^{k-2}}+\eps \br^2|V_2|\prod_{i=3}^{k}|V_i|^2\,,
\]
which in combination with~\eqref{eq:intro-a} leads to
\[
a< \bl\tfrac 1{2^{k-2}}+\eps \br M\,.
\]
Plugging this into~\eqref{eq:more-a} we learn
\begin{align*}
	\sum_{x\in V_1}g(x)^2 
	&<\left(\bl \tfrac 14+\xi\br \bl\tfrac 1{2^{k-2}}
	+\eps \br^2+\xi\right) |V_1|\,|V_2|^2\, M^2\\
	&=\left( (1+4\xi)\bl\tfrac 1{2^{k-1}}+\tfrac\eps 2 \br^2+\xi\right) |V_1|\,|V_2|^2\, M^2
\end{align*}
and using the choice of $\xi$ again we obtain the desired conclusion.
\end{proof}

The following is easy by now.

\begin{proof}[Proof of Theorem~\ref{thm:path}]
Given $\eps$ and $k$ we take $\xi$ to be the number delivered by the foregoing proposition. 
Consider a $k$-partite graph $G$ with vertex classes $V_1, \ldots, V_{k}$ such that 
$G[V_r, V_{r+1}]$ is $\xi$-poor for all $r\in[k-1]$. Let the function $g$ be defined as in 
Proposition~\ref{prop:strong-path}. Then the number of $k$-vertex paths in $G$ we are to 
bound from above may be written as $\sum_{x\in V_1}g(x)$. Now we have just proved 
\[
	\sum_{x\in V_1}g(x)^2< \bl\tfrac 1{2^{k-1}}+\eps\br^2|V_1|\prod_{i=2}^{k}|V_i|^2 
\]
and in view of the inequality
\[
	\left(\sum_{x\in V_1}g(x)\right)^2\le |V_1| \sum_{x\in V_1}g(x)^2	
\]
this yields indeed
\[
	\sum_{x\in V_1}g(x)< \bl\tfrac 1{2^{k-1}}+\eps\br\prod_{i=1}^{k}|V_i|\,. \qedhere
\]
\end{proof}

\section{The triangle lemma} 
\label{sec:triangle}

The last promise we need to fulfill is to prove Theorem~\ref{thm:triangle}.
This will in turn be prepared by the following statement.

\begin{lemma} \label{13}
Given a real number $\delta\in(0, 1)$ and integers $m\ge k\ge 0$ there exists a 
positive integer $M=F(\delta, k,m)$ with the following property: 
suppose that we have 
\begin{enumerate}[label=\rmlabel]
\item\label{it:61i} finite nonempty sets $A_1, \ldots, A_M$,
\item\label{it:61ii} and subsets $X_{ij}\subseteq A_i$ with 
                 $| X_{ij}| \ge \delta\,|A_i|$ for $1\le i<j\le M$.
\end{enumerate}
Then there are indices $1\le n_1<\ldots<n_m\le M$ and elements 
$a_1\in A_{n_1}, \ldots, a_k\in A_{n_k}$ such that 
\[
	a_i\in \bigcap_{j\in (i, m]}X_{n_i n_j}
\]
holds for all $i\in [k]$.  
\end{lemma}

\begin{proof}
We argue by induction on $k$. In the base case $k=0$ we set $F(\delta, 0, m)=m$. 
Then we may always take $n_i=i$ for all $i\in [m]$ because there are no further 
choices to make or conditions to meet. 

Suppose that the result is already known for some integer $k$ and all relevant 
combinations of $\delta$ and $m$. 
Now if a a real number $\delta\in(0,1)$ and an integer $m\ge k+1$ are given,
we set
\[
	m'=k+1+\left\lceil \frac{m-k-1}{\delta}\right\rceil 
	\quad \text{ and then } \qquad 
	M=F(\delta, k+1, m)=F(\delta, k, m')\,.
\]

Intending to verify that $M$ has the desired property, we consider any sets 
$A_i$ and~$X_{ij}$ obeying the above clauses~\ref{it:61i} and~\ref{it:61ii}. 
Owing to the definition of~$M$, there exist indices $1\le n_1<\ldots<n_{m'}\le M$ and 
elements $a_1\in A_{n_1}, \ldots, a_k\in A_{n_k}$ such that 
\[
	a_i\in  \bigcap_{j\in (i, m']} X_{n_i n_j}
\]
holds for all $i\in [k]$. 
The estimates from~\ref{it:61ii} yield
\[
	(m'-k-1) |A_{n_{k+1}}|\,\delta\le \sum_{j=k+2}^{m'}|X_{n_{k+1} n_{j}}|\,.
\]
So by double counting there is an element $a_{k+1}\in A_{n_{k+1}}$ for which the set 
\[
	Q=\{j\in [k+2, m']\colon  a_{k+1}\in X_{n_{k+1} n_{j}}\}
\]
satisfies $| Q| \ge \delta(m'-k-1)$. By our choice of $m'$ this implies
$| Q| \ge m-k-1$ and thus we may select some numbers~ 
$\ell(k+2)< \ldots< \ell(m)$ from $Q$. 
Now it is not hard to check that the indices 
$n_1<\ldots <n_{k+1}<n_{\ell(k+2)}< \ldots< n_{\ell(m)}$ 
as well as the elements $a_1, \ldots, a_{k+1}$ satisfy the conclusion of our lemma.
\end{proof}

We may now conclude the proof of our main result by showing the triangle lemma.

\begin{proof}[Proof of Theorem~\ref{thm:triangle}] 
For notational reasons it is slightly preferable to assume that for $1\le i<j\le m$ 
the graph $G[V_j, V_i]$ rather than $G[V_i, V_j]$ is $\xi$-rich. This change of hypothesis 
is allowed by symmetry, i.e., since we may read the original sequence of sets 
$V_1, \ldots, V_m$ backwards. It will also be convenient to write $G_{ij}$ in place of 
$G[V_i, V_j]$ whenever $1\le i<j\le m$. 

Now the assumption means that for $1\le i<j\le m$ the set $X_{ij}$ consisting of all 
those vertices $v\in V_i$ at which more than $\bl \tfrac14+\xi\br |V_i|\,|V_j|$ 
two-edge walks of $G_{ij}$ start satisfies $|X_{ij}|>\xi\,|V_i|$. 

The arguments that follow will rely on the hierarchy 
\[
	m^{-1}\ll m_*^{-1}\ll m_{**}^{-1}\ll \delta\ll \xi\,,
\]
where for transparency we assume that $\delta^{-1}$ is an integer.
The first step is to apply the previous lemma, using $m\ge F(\xi, m_*, m_*)$. 
Upon a relabeling of indices this yields some vertices $a_i\in V_i$ for $1\le i\le m_*$ 
such that 
\[
	a_i\in\bigcap_{j\in (i, m_*]}X_{ij}
\]
holds for all $i\in [m_*]$. As we shall see, there is a triangle in $G$ whose vertices 
are from $V_1\cup\ldots\cup V_{m_*}$.

Next we consider a function 
\[
	t\colon [m_*]^{(2)}\longrightarrow \bigl[\delta^{-1}\bigr]\,,
\]
with the property that for $1\le i<j\le m_*$ the integer $t=t(i,j)$ satisfies
\[
	| N(a_i)\cap V_j| \in [t, t+1]\cdot \delta\,|V_j|\,.
\]
Ramsey's theorem allows us to assume by another relabeling of indices that $t$ is constant 
on $[m_{**}]^{(2)}$, attaining always the same value $t_*$, say. From now on we intend to 
exhibit a triangle with two vertices from $V_1\cup\ldots\cup V_{m_{**}-1}$ and one vertex 
from $V_{m_{**}}$.

For this purpose, we will consider for $1\le j<m_{**}$ the sets
\[
	A_j=N(a_j)\cap V_{m_{**}} \quad \text{ and } \quad B_j=A_j-\bigcup_{1\le i<j}A_i\,.
\]
Since $B_1, B_2, \ldots, B_{m_{**}-1}$ are mutually disjoint subsets  of $V_{m_{**}}$ and 
$m_{**}\gg \delta^{-1}$, there is an index $j_*$ with 
$| B_{j_*}|\le \delta\,| V_{m_{**}}|$. 

In order to find the desired triangle we will first assume that there exists an 
index $i_*<j_*$ together with a vertex $x\in A_{i_*}$ such that 
$|N(x)\cap V_{j_*}| > (1-t_*\delta)|V_{j*}|$ holds. 
Due to the choice of $t_*$ we also have 
$|N(a_{i_*})\cap V_{j_{*}}| \ge t_*\,\delta\,| V_{j_{*}}|$. 
The addition of both estimates yields
\[
	|N(x)\cap V_{j_*}| +| N(a_{i_*})\cap V_{j_{*}}| > |V_{j_{*}}|\,,
\]
and thus there is a common neighbour $y\in V_{j_*}$ of $a_{i_{*}}$ and $x$. 
Now $a_{i_{*}}x$ is an edge of $G$ as well, because $x\in A_{i_*}$. 
So altogether $a_{i_*}xy$ is a triangle in $G$.

To finish the argument we will now prove that indeed there always exists a vertex 
${x\in \bigcup_{1\le i<j_{*}}A_i}$ with $| N(x)\cap V_{j_*}|> (1-t_*\delta)|V_{j*}|$. 
If this were not the case, we could estimate the number $\Omega$ of two-edge walks 
in $G_{j_*m_{**}}$ that start at $a_{j_*}$ by
\[
	\Omega=\sum_{x\in A_{j_*}}| N(x)\cap V_{j^*}| \le 
	| A_{j_*}| \cdot (1-t_*\delta)| V_{j*}| 
	+| B_{j_*}| \cdot |V_{j_{*}}| \,.
\]
Because of $|A_{j_*}|\le (t_*+1)\delta\,|V_{m_{**}}|$ and $|B_{j_*}|\le \delta\,|V_{m_{**}}|$,
this leads to
\[
	\Omega\le \bigl((t_*+1)\delta\cdot(1-t_*\delta)+\delta\bigr)|V_{j_*}|\,|V_{m_{**}}|\,.
\]
On the other hand $a_{j_*}\in X_{j_*m_{**}}$ implies 
$\Omega>\bl\tfrac14+\xi\br|V_{j_*}|\,|V_{m_{**}}|$,  
so that altogether we obtain
\[
\tfrac14+\xi< (t_*+1)\delta\cdot(1-t_*\delta)+\delta\,.
\]
But in view of $t_*\delta\cdot(1-t_*\delta)\le \tfrac14$ this entails 
$\xi<2\delta$, which contradicts the hierarchy imposed above. 
\end{proof}
 
\section{Concluding Remarks}
\label{sec:conclude}

\subsection{An ordered version of the three edge theorem}
In~\cite{RRS-a} we actually obtained slightly more than just 
$\pivvv\bl K_4^{(3)-}\br\le\tfrac 14$. We also proved that for $n^{-1}\ll\eta\ll\eps$
every $\bl\tfrac 14+\eps, \eta, \VVV\br$-dense $3$-uniform hypergraph with an ordered vertex 
set of size $n$ contains a~$K_4^{(3)-}$ whose vertex of degree $3$ occurs either in the 
first or last position. In other words, its three vertices of degree $2$ appear 
consecutively. More generally, an $F^{(k)}$ has three vertices of degree $2$, while all other 
vertices have degree $3$, and our proof of Theorem~\ref{thm:main} can be modified to show 
the following result.

\begin{thm}
For $n^{-1}\ll\eta\ll\eps$ every $(2^{1-k}+\eps, \eta, k-2)$-dense $k$-uniform hypergraph $H$
with vertex set $[n]$ contains an $F^{(k)}$ with the additional property that its three vertices
of degree $2$ appear at consecutive positions.
\end{thm}

The key observation one needs for showing this is that in the proof of 
Proposition~\ref{prop:reduced} the indices $t_1$, $t_2$, and $t_3$ 
appear consecutively in the increasing 
enumeration of $z^{-}\cup\{t_1, t_2, t_3\}$. To make use of this fact, we need to 
start from a regular partition of $H$ whose vertex partition refines a partition 
into many consecutive intervals, and the iterated refinement strategy on which the 
proof of the hypergraph regularity lemma relies 
allows us to obtain this. The full argument would be very similar 
to~\cite{RRS-a} and we leave the details to the reader.

\subsection{Relaxing the density condition}\label{subsec:exmp}
Generalising a construction due to Leader and Tan~\cite{LeTa10} we will now prove 
that at least when $k$ is divisible by $4$ the three edge theorem cannot be improved
by replacing $\pi_{k-2}$ by $\pi_{k-3}$.

\begin{prop}\label{prop:4k}
If $4\mid k$, then $\pi_{k-3}\bigl(F^{(k)}\bigr)\ge 2^{2-k}$.
\end{prop}

\begin{proof}
Consider a $(k-2)$-uniform tournament $T_n^{(k-2)}$ with vertex set $[n]$. 
We define the $(k-1)$-uniform tournament $DT_n^{(k-1)}$ with each $x\in [n]^{(k-1)}$
receiving that orientation $\sigma$ which has the property that the number of
elements $i\in x$ for which $T_n^{(k-2)}$ assigns the orientation $\sigma_i$ to $x\setminus\{i\}$ is even.
Notice that this conditions determines uniquely which of the two possible orientations
$DT_n^{(k-1)}$ assigns to $x$ because $k-1$ is odd.

By Fact~\ref{fact:Fk-free} the hypergraph $H\bigl(DT_n^{(k-1)}\bigr)$ is always $F^{(k)}$-free,
so it suffices to prove that if $T_n^{(k-2)}$ gets chosen uniformly at random, then for any fixed 
$\eta>0$ the probability that~$H\bigl(DT_n^{(k-1)}\bigr)$ is $(2^{2-k}, \eta, k-3)$-dense 
approaches $1$ as $n$ tends to infinity. This can be shown by the same strategy as 
Lemma~\ref{lem:Hdense} provided that one knows that any fixed  $e\in [n]^{(k)}$ has a probability
of $2^{2-k}$ to be an edge of $H\bigl(DT_n^{(k-1)}\bigr)$. 

By symmetry we only need to prove this for $e=[k]$ and $n=k$. 
Let $\sigma$ be the orientation $+(1, 2, \ldots, k)$ of $[k]$ and denote the event that
$DT_n^{(k-1)}$ assigns for every $i\in [k]$ the orientation $\sigma_i$ to $[k]\setminus\{i\}$ by $\ccE$.
As proved below, we have
\begin{equation}\label{eq:prob-E}
	\PP(\ccE)=2^{1-k}\,.
\end{equation}
By symmetry the corresponding statement about $-\sigma$ holds as well and taken together 
these two equations show that $[k]$ has indeed a probability of $2\cdot 2^{1-k}$ of being 
an edge of $H\bigl(DT_k^{(k-1)}\bigr)$. Thus the proof of~\eqref{eq:prob-E}
concludes at the same time the proof of Proposition~\ref{prop:4k}.

Before we proceed to the proof of~\eqref{eq:prob-E} we associate a bipartite graph $GT_k^{(k-2)}$
with any $(k-2)$-uniform tournament $T_k^{(k-2)}$. Its two vertex classes are the set 
$AT_k^{(k-2)}$ of orientations which $T_k^{(k-2)}$ associates to the members of $[k]^{(k-2)}$
and the set $B=\{\sigma_i \colon i\in [k]\}$. An edge between $a\in AT_k^{(k-2)}$ and $\sigma_i\in B$
signifies that $a$ is an orientation of a subset of $[k]\setminus\{i\}$ that is induced by $\sigma_i$.
Whenever $1\le i<j\le k$ the orientations $\sigma_i$ and $\sigma_j$ induce different orientations
on the $(k-2)$-set $[k]-\{i, j\}$ and consequently every $a\in AT_k^{(k-2)}$ has degree $1$ 
in~$GT_k^{(k-2)}$. The total number of edges of $GT_k^{(k-2)}$ is therefore $\binom{k}{2}$ and, 
as $k$ is a multiple of~$4$, it follows that 
\begin{equation}\label{eq:E-even}
	GT_k^{(k-2)} \text{ has an even number of edges}\,.
\end{equation}
Due to the definition of $DT_k^{(k-1)}$, the event $\ccE$ happens if and only if 
every vertex $\sigma_i\in B$ has even degree, which by~\eqref{eq:E-even} is equivalent 
to the vertices from $B-\{\sigma_k\}$ having even degrees.

Now let $\alpha$ be any assignment of orientations to the members of 
\[
	\bigl\{x\in [k]^{(k-2)}\colon k\in x\bigr\}\,.
\]
In order to prove~\eqref{eq:prob-E} it suffices to show that the conditional probability 
of $\ccE$ given that~$T_k^{(k)}$ extends $\alpha$ is $2^{1-k}$. Given $\alpha$ the only 
information about $T_k^{(k-1)}$ we still need for figuring out whether $\ccE$ holds are the
orientations of the sets $[k-1]\setminus\{i\}$ with $i\in [k-1]$. Moreover for each $i\in [k-1]$
there is a unique way of orienting $[k-1]\setminus\{i\}$ in such a way that $\sigma_i$ receives
an even degree in $GT_k^{(k-1)}$ and the probability that $T_k^{(k-1)}$ orients $[k-1]\setminus\{i\}$
in this manner is $\tfrac12$. Hence given $\alpha$ the probability that $\ccE$ holds, 
i.e., that all vertices from $B-\{\sigma_k\}$ have even degrees, is indeed $2^{1-k}$.
\end{proof}  

\subsection{More edges} 
One of the perhaps most important conjectures about generalised Tur\'an densities of 
$3$-uniform hypergraphs states that $\pivvv\bl K^{(3)}_4\br=\tfrac 12$. The lower 
bound follows from a construction presented by R\"odl in~\cite{Ro86} and the most 
recent contribution in favour of this conjecture seems to be the formula
$\piev\bl K^{(3)}_4\br=\tfrac 12$ obtained in~\cite{RRS-b}.

R\"odl's construction and the random tournament hypergraph admit a common generalisation.
Depending on any colouring
\[
	\gamma\colon [n]^{(k-1)} \longrightarrow\{\text{red}, \text{green}\}
\]
and any integer $r\in [3, k+1]$ we define a $k$-uniform hypergraph $H^{(k)}_r(\gamma)$
with vertex set~$[n]$ in the following way: if $\{x_1, \ldots, x_k\}$ lists the elements of 
some $x\in [n]^{(k)}$ in increasing order, then $x$ is declared to be an edge of 
$H^{(k)}_r(\gamma)$ if and only if 
\[
	\gamma(x\setminus\{x_1\})\ne \gamma(x\setminus\{x_2\})\ne \ldots \ne \gamma(x\setminus\{x_{k+3-r}\})\,.
\]
These hypergraphs can be used to obtain the following lower bound:

\begin{fact}\label{fact:Fkr}
If $k\ge 2$ and $3\le r\le k+1$, then the $k$-uniform hypergraph $F^{(k)}_r$ with $(k+1)$
vertices and $r$ edges satisfies $\pi_{k-2}\bl F^{(k)}_r\br\ge 2^{r-k-2}$.
\end{fact} 

\begin{proof}
An argument very similar to the proof of Lemma~\ref{lem:Hdense} shows that for fixed $\eta$
the probability that $H^{(k)}_r(\gamma)$ is $(2^{r-k-2}, \eta, k-2)$-dense tends to $1$ as 
$n$ tends to infinity. 

Thus it suffices to prove that for no colouring $\gamma$ the hypergraph
$H^{(k)}_r(\gamma)$ can contain an~$F^{(k)}_r$. We verify this by induction on $r$.
To deal with the base case $r=3$ it suffices in view of Fact~\ref{fact:Fk-free} to observe 
that $H^{(k)}_r(\gamma)=H\bl T_n^{(k-1)}\br$, where the higher order tournament $T_n^{(k-1)}$
is defined as follows: if $y=\{y_1, \ldots, y_{k-1}\}$ lists the elements of some 
$y\in [n]^{(k-1)}$ in increasing order, then $y$ receives the orientation 
$+(y_1, \ldots, y_{k-1})$ in $T_n^{(k-1)}$ if $\gamma(y)=\text{red}$, and otherwise it 
receives the opposite orientation.  

For the induction step from $r$ to $r+1$ we assume that for some colouring $\gamma$
of $[n]^{(k-1)}$ there would exist an $F^{(k)}_{r+1}$ in $H^{(k)}_{r+1}(\gamma)$,
say with vertices $v_1<v_2<\ldots, v_{k+1}$. Observe that $k+1\ge r+1\ge 4$ yields $k\ge 3$.
Moreover, at least $r$ edges of our $F^{(k)}_{r+1}$ must contain $v_{k+1}$.
Thus if we set $\overline{n}=v_{k+1}-1$ and let 
$\overline{\gamma}\colon [\overline{n}]^{(k-2)}\longrightarrow \{\text{red}, \text{green}\}$ 
be the colouring defined by 
$\overline{\gamma}(z)=\gamma(z\cup\{v_{k+1}\})$ for all $z\in [\overline{n}]^{(k-2)}$,
then $\{v_1, \ldots, v_{k}\}$ spans an $F^{(k-1)}_{r}$ in $H^{(k-1)}_{r}(\overline{\gamma})$,
contrary to the induction hypothesis.
\end{proof}

It would be extremely interesting if the lower bound just obtained were optimal.
Notice that this holds for $r=3$ owing to the three edge theorem while the case 
$k=3$ and $r=4$ corresponds to the problem of deciding whether 
$\pivvv\bl K^{(3)}_4\br=\tfrac 12$ holds mentioned above.

In the special case $r=k+1$ we get the lower bound $\pi_{k-2}\bl K^{(k)}_{k+1}\br\ge\tfrac 12$
addressing the generalised Tur\'{a}n density of the clique with $k+1$ vertices. 
The construction given in the proof of Fact~\ref{fact:Fkr}
showing this lower bound extends to larger cliques as follows.

\begin{fact}\label{fact:cliques}
If $t\ge k\ge 2$, then $\pi_{k-2}\bl K^{(k)}_t\br\ge \tfrac{t-k}{t-k+1}$.
\end{fact}

\begin{proof}
Depending on any colouring 
\[
	\phi\colon [n]^{(k-1)}\longrightarrow [t-k+1]
\]
we define a $k$-uniform hypergraph $R^{(k)}(\phi)$ with vertex set $[n]$ having all
those $k$-sets $\{v_1, \ldots, v_k\}$ with $v_1<\ldots< v_k$ as edges that satisfy
$\phi(\{x_2, \ldots, x_k\})\ne \phi(\{x_1, x_3,\ldots, x_k\})$.

Again standard arguments show that the probability for $R^{(k)}(\phi)$ to 
be $\bl\tfrac{t-k}{t-k+1}, \eta, k-2\br$-dense tends for fixed $\eta$ to $1$ as 
$n$ tends to infinity. 

Assume for the sake of contradiction that some $t$ vertices, say $v_1<\ldots< v_t$,
would span a clique in $R^{(k)}(\phi)$. Let $z=\{v_{t+3-k}, \ldots, v_t\}$ denote 
the set of the last $k-2$ vertices of this clique. Due to the so-called 
\emph{Schubfachprinzip} (also known as pigeonhole principle)  there 
must be two indices $1\le i<j\le t+2-k$ with $\phi(\{v_i\}\cup z)=\phi(\{v_j\}\cup z)$.
But this means that $\{v_i, v_j\}\cup z$ cannot be an edge of~$R^{(k)}(\phi)$.   
\end{proof}

It may be interesting to observe that by Tur\'an's theorem Fact~\ref{fact:cliques} 
holds with equality for $k=2$. We are not aware of any construction showing that this
cannot be true in general.

\begin{quest} 
Do we have $\pi_{k-2}\bl K^{(k)}_t\br= \tfrac{t-k}{t-k+1}$ whenever $t\ge k\ge 2$?
\end{quest} 

Notice that for $k=3$ and $t=6$ there is a construction demonstrating 
$\pivvv\bl K^{(3)}_6\br \ge \tfrac34$ different from the above one described
in~\cite{RRS-a}*{Subsection~5.1}.

\end{document}